\documentclass{amsart}

\usepackage{tikz}
\usetikzlibrary{matrix,arrows}
\usepackage{amsmath, amsthm, amssymb}
\usepackage{hyperref}

\DeclareFontFamily{U}{mathx}{\hyphenchar\font45}
\DeclareFontShape{U}{mathx}{m}{n}{
      <5> <6> <7> <8> <9> <10>
      <10.95> <12> <14.4> <17.28> <20.74> <24.88>
      mathx10
      }{}
\DeclareSymbolFont{mathx}{U}{mathx}{m}{n}
\DeclareFontSubstitution{U}{mathx}{m}{n}
\DeclareMathAccent{\widecheck}{0}{mathx}{'161}

\usepackage{verbatim}
\newcommand{\Po}{\mathcal{P}}

\newcommand{\im}{\mathrm{im}}
\newcommand{\F}{\mathrm{F}}
\newcommand{\I}{\mathrm{I}}

\newcommand{\embeds}{\rightarrowtail}
\newcommand{\op}{\mathrm{op}}
\renewcommand{\hat}{\widehat}
\renewcommand{\check}{\widecheck}
\newcommand{\id}{\mathrm{id}}
\newcommand{\gen}[1]{\langle #1 \rangle}

\theoremstyle{plain}
\newtheorem{thm}{Theorem}

\newtheorem{lem}[thm]{Lemma}
\newtheorem{cor}[thm]{Corollary}
\newtheorem{prop}[thm]{Proposition}

\theoremstyle{definition}
\newtheorem{dfn}[thm]{Definition}
\newtheorem{exa}[thm]{Example}

\newtheorem{rem}[thm]{Remark}

\numberwithin{thm}{section}

\newcommand{\aidl}{\mathrm{aidl}}

\newcommand{\Lwav}{\mathbf{L_{\wedge,a\vee}}}
\newcommand{\Lvaw}{\mathbf{L_{\vee,a\wedge}}}

\newcommand{\lincol}{black}
\newcommand{\linth}{thick}
\newcommand{\po}[2][\pocol]{\filldraw[#1](#2) circle (2 pt);}
\newcommand{\li}[1]{\draw[\linth,\lincol] #1;}

\title[Distributive envelopes and topological duality for lattices]{Distributive envelopes and \\ topological duality for lattices \\ via canonical extensions}
\author[M. Gehrke]{Mai Gehrke}
\address{MG: LIAFA, CNRS and Universit\'e Paris Diderot, Paris 7\newline
Case 7014
F-75205 Paris Cedex 13\\
France}
\email{mgehrke@liafa.univ-paris-diderot.fr}

\author[S. J. van Gool]{Samuel J. van Gool}
\address{SVG: IMAPP, Radboud Universiteit Nijmegen and LIAFA, Universit\'e Paris Diderot, Paris 7 \newline
P.O. Box 9010
6500 GL Nijmegen\\
The Netherlands}
\email{samvangool@me.com}

\thanks{The authors would like to thank Andrew Craig and Hilary Priestley for their input in the initial stages of the project leading to this paper. The first-named author would like to thank ANR 2010 BLAN 0202 02 FREC for partial support during the later stages of writing this paper. 
The PhD research project of the second-named author has been made possible by NWO grant 617.023.815 of the Netherlands Organization for Scientific Research (NWO). The authors also thank the organizations of the conference TACL 2011 (Marseille) and the workshop Duality `12 (Oxford) for giving them the opportunity to present their work there.
}

\begin{document}

\begin{abstract}We establish a topological duality for bounded lattices. The two main features of our duality are that it generalizes Stone duality for bounded distributive lattices, and that the morphisms on either side are not the standard ones. A positive consequence of the choice of morphisms is that those on the topological side are functional. 

\noindent Towards obtaining the topological duality, we develop a universal construction which associates to an arbitrary lattice two distributive lattice envelopes with a Galois connection between them. This is a modification of a construction of the injective hull of a semilattice by Bruns and Lakser, adjusting their concept of `admissibility' to the finitary case. 

\noindent Finally, we show that the dual spaces of the distributive envelopes of a lattice coincide with completions of quasi-uniform spaces naturally associated with the lattice, thus giving a precise spatial meaning to the distributive envelopes.
\end{abstract}

\maketitle

% SECTION INTRODUCTION

\section{Introduction}
 \parindent 0pt	
Topological duality for Boolean algebras \cite{Sto1936} and distributive lattices \cite{Sto1937} is a useful tool for studying relational semantics for propositional logics. Canonical extensions \cite{Jon1951, Jon1952, Geh1994, GehHar2001} provide a way of looking at these semantics algebraically. In the absence of a satisfactory topological duality, canonical extensions have been used \cite{DGP2005} to treat relational semantics for substructural logics. The relationship between canonical extensions and topological dualities in the distributive case suggests that canonical extensions should be taken into account when looking for a topological duality for arbitrary bounded\footnote{From here on, we will drop the adjective `bounded', adopting the convention that all lattices considered in this paper are bounded.} lattices. The main aim of this paper is to investigate this line of research. 

Several different approaches to topological duality for lattices exist in the literature, starting from Urquhart \cite{Urq1978}. Important contributions were made, among others, by Hartung \cite{Hartung1992,Hartung1993}, who connected Urquhart's duality to the Formal Concept Analysis \cite{GanWil1999} approach to lattices. However, as we will show in Section~\ref{sec:duality} of this paper, a space which occurs in Hartung's duality can be rather ill-behaved. In particular, such a space need not be sober, and therefore it need not occur as the Stone dual space of any distributive lattice. By contrast, the spaces that occur in the duality in Section~\ref{sec:duality} of this paper are Stone dual spaces of certain distributive lattices that are naturally associated to the given lattice. In topological duality for lattices \cite{Hartung1993}, the morphisms in the dual category are necessarily relational rather than functional. In this paper, we exhibit a class of lattice morphisms for which the morphisms in the dual category can still be functional.

In order to obtain these results, we first develop a relevant piece of order theory that may be of independent interest. The ideas that play a role here originate with the construction of the injective hull of a semilattice \cite{Bru1970}, which, as it turns out, is a frame. In Section~\ref{sec:envelopes}, we recast this construction in the finitary setting to obtain a construction of a pair of distributive lattices from a given lattice, which we shall call the {\it distributive envelopes} of the lattice. Moreover, as we will also see in Section~\ref{sec:envelopes}, these two distributive envelopes correspond to the meet- and join-semilattice-reducts of the lattice of departure and are linked by a Galois connection whose lattice of Galois-closed sets is isomorphic to the original lattice. In Section~\ref{sec:duality}, we then use Stone-Priestley duality for distributive lattices \cite{Sto1937, Pri1970} and the theory of canonical extensions to find an appropriate category dual to the category of lattices, using this representation of a lattice as a pair of distributive lattices with a Galois connection between them. Particular attention is devoted to morphisms; the algebraic results from Section~\ref{sec:envelopes} will guide us towards a notion of `admissible morphism' between lattices, which have the property that their topological duals are functional. Finally, in Section~\ref{sec:uniform}, we will propose quasi-uniform spaces as an alternative to topology for studying set representations of lattices.

% SECTION ON CANONICAL EXTENSIONS AND PREVIOUSLY EXISTING DUALITIES

\section{Canonical extensions}
\label{sec:canext}
Canonical extensions \cite{Jon1951,Geh1994} provide an algebraic version of the Stone and Priestley dualities for distributive lattices. The algebraic characterization of the canonical extension as a particular lattice completion led to its generalization to arbitrary lattices in \cite{GehHar2001}. Recently, a ``canonical envelope'' for spatial preframes was constructed \cite[Section 9]{Ern2007}, which, in the case where the preframe is a coherent frame, coincides with the canonical extension of the distributive lattice of compact elements of that frame. In this section we recall the basic facts about canonical extensions that are relevant to this paper; our main reference will be \cite{GehHar2001}. 

\begin{dfn}
Let $L$ be a lattice. A {\it canonical extension} of $L$ is an embedding $e : L \embeds C$ of $L$ into a complete lattice $C$ satisfying
\begin{enumerate}
\item for all $u \in C$, we have 
\[ \bigvee \left\{\bigwedge e[S] \ | \ S \subseteq L, \bigwedge e[S] \leq u\right\} = u = \bigwedge \left\{\bigvee e[T] \ | \ T \subseteq L, u \leq \bigvee e[T]\right\};\] 
\item for all $S, T \subseteq L$, if $\bigwedge e[S] \leq \bigvee e[T]$ in $C$, then there are finite $S' \subseteq S$ and $T' \subseteq T$ such that $\bigwedge S' \leq \bigvee T'$ in $L$.
\end{enumerate}
The first property is commonly referred to as {\it denseness}, the second as {\it compactness}.
\end{dfn}

\begin{thm}\label{thm:canextexun}
Let $L$ be a lattice. There exists a canonical extension $e : L \embeds C$. Moreover, if $e : L \embeds C$ and $e' : L \embeds C'$ are canonical extensions of $L$, then there is a complete lattice isomorphism $\phi : C \to C'$ such that $\phi \circ e = e'$.
\end{thm}
\begin{proof}
%[Proof (Sketch).] 
%Let $P_L := \gen{\F(L), \I(L), R}$ be the polarity consisting of the filters of $L$, the ideals of $L$, with the relation $R$ defined by $F\, {R}\, I$ if, and only if, $F \cap I \neq \emptyset$. One may now show that $(P_L)^+$ is a canonical extension of $L$, and that any canonical extension of $L$ is isomorphic to $(P_L)^+$. We refer the reader to \cite[Prop. 2.6 and 2.7]{GehHar2001} for more details.
See \cite[Prop. 2.6 and 2.7]{GehHar2001}.
\end{proof}
%Consider the Galois connection $u : \Po(\F(L)) \leftrightarrows \Po(\I(L)) : l$ given by 
%$$u(A) := \{I \in \I(L) \ | \ \forall F \in A : F \cap I \neq \emptyset\} \quad \quad (A \subseteq \F(L)),$$
%$$l(B) := \{F \in \F(L) \ | \ \forall I \in B : F \cap I \neq \emptyset\} \quad \quad (B \subseteq \I(L)).$$ 
%Let $C$ be the complete lattice of Galois-closed sets and let $e : L \to C$ be the map sending $a \in L$ to $e(a) := \{F \in \F(L) \ | \ a \in F\}$. One may check that $e$ is a well-defined dense and compact embedding.
%\end{proof}
As is common in the literature and justified by this theorem, we will speak of {\it the} canonical extension of a lattice $L$ and we will denote it by $L^\delta$. We will also omit reference to the embedding $e$, and regard $L$ as a sublattice of $L^\delta$. The closure of $L$ under infinite meet inside $L^\delta$ is isomorphic to the filter completion $\F(L)$ of $L$, and its elements are therefore known as the {\it filter elements} of $L^\delta$. Similarly, the join-closure of $L$ inside $L^\delta$ is isomorphic to the ideal completion $\I(L)$ of $L$ and consists of the {\it ideal elements} of $L^\delta$. The elements of $L$ are characterized in $L^\delta$ as exactly those which are both filter and ideal elements. See \cite[Lemma 3.3]{GehHar2001} for proofs of the facts mentioned in this paragraph.

To end this preliminary section we recall one more important fact, the proof of which relies on the axiom of choice. This proposition states that the canonical extension of a lattice `has enough points' and therefore it is crucial for obtaining set-theoretic representations of lattices via canonical extensions.
\begin{prop}[Canonical extensions are perfect lattices]\label{prop:canextperfect}
Let $L$ be a lattice. The set of completely join-irreducible elements $J^\infty(L^\delta)$ of the canonical extension $\bigvee$-generates $L^\delta$, and the set of completely meet-irreducible elements $M^\infty(L^\delta)$ of the canonical extension $\bigwedge$-generates in $L^\delta$.
\end{prop}
\begin{proof}
See \cite[Lemma 3.4]{GehHar2001}.
\end{proof}

% SECTION ON DISTRIBUTIVE ENVELOPES

\section{Distributive envelopes}\label{sec:envelopes}
In this section we introduce the two {\it distributive envelopes} $D^\wedge(L)$ and $D^\vee(L)$ of a lattice $L$. After giving the universal property defining the envelope, we will give both a point-free and a point-set construction of it, and investigate the categorical properties of the envelopes. Finally, we will show how the two distributive envelopes are linked by a Galois connection which enables one to recover the original lattice $L$. Some of the results in this section can be seen as finitary versions of the results on injective hulls of semilattices of Bruns and Lakser \cite{Bru1970}. We will relate our results to theirs in Remark~\ref{rem:comparebrunslakser}. However, the reader who is not familiar with \cite{Bru1970} should be able to read this section independently.

The following definition is central, being the finitary version of the definition of {\it admissible} given in \cite{Bru1970}.
%SVG Seeing it this way, I think it is actually much more natural to do the construction for meet-semilattices instead of lattices, then remark that of course you can do the same thing for join-semilattices, and then for lattices you can do both. I don't think it's worth it to rewrite this section in this way, but I would like to ultimately do that in my thesis.
\begin{dfn}
Let $L$ be a lattice. A finite subset $M \subseteq L$ is {\it join-admissible} if its join distributes over all meets with elements from $L$, i.e., if, for all $a \in L$,
\begin{equation*}
 a \wedge \bigvee M = \bigvee_{m \in M} (a \wedge m).
\end{equation*}
We say that a function $f : L_1 \to L_2$ between lattices {\it preserves admissible joins} if, for each finite join-admissible set $M \subseteq L_1$, we have $f(\bigvee M) = \bigvee_{m \in M} f(m)$.
\end{dfn}
The formal definition of the distributive $\wedge$-envelope is now as follows.
\begin{dfn}\label{def:distenv}
Let $L$ be a lattice. An embedding $\eta^\wedge_L : L \embeds D^\wedge(L)$ of $L$ into a distributive lattice $D^\wedge(L)$ which preserves meets and admissible joins is a {\it distributive $\wedge$-envelope} of $L$ if it satisfies the following universal property:
\begin{align*}
&\text{\it For any function $f : L \to D$ into a distributive lattice $D$ that preserves} \\
&\text{\it finite meets and admissible joins, there exists a unique lattice homomorphism}\\
&\text{\it $\hat{f} : D^\wedge(L) \to D$ such that $\hat{f} \circ \eta^\wedge_L = f$, i.e., the following diagram commutes:}
\end{align*}

\vspace{-4mm}

\begin{center}
\begin{tikzpicture}
\matrix (m) [matrix of math nodes, row sep=2.5em, column sep=2.5em, text height=1.5ex, text depth=0.25ex] 
{L & D^\wedge(L) \\
    &  D \\};

\path[->] (m-1-1) edge node[below] {$f$} 
%node[left] {$\wedge a\vee$} 
(m-2-2); 
\draw[->,dashed] (m-1-2) edge node[right] {$\hat{f}$} node[left] {$!$} (m-2-2);
\path[>->] (m-1-1) edge node[auto] {$\eta^\wedge_L$} (m-1-2);
\end{tikzpicture}
\end{center}

\vspace{-2mm}

The definition of the distributive $\vee$-envelope $D^\vee(L)$ of $L$ is order dual, cf. Remark~\ref{rem:orderdual} below. 
\end{dfn}
Let us give some intuition for the above definitions. The join-admissible subsets of $L$ are those subsets whose joins `are already distributive' in $L$. A distributive $\wedge$-envelope of a lattice $L$ is a universal solution to the question of embedding $L$ as a $\wedge$-semilattice into a distributive lattice while preserving all admissible joins. Clearly, a non-admissible join can not be preserved by any $\wedge$-embedding into a distributive lattice; in this sense, a distributive $\wedge$-envelope adds `as few joins as possible' to make $L$ distributive.

The main aim of this section is to show that the distributive $\wedge$-envelope of a lattice always exists (Theorem~\ref{thm:envelopeexists}); it is then clearly unique up to isomorphism. The same results of course hold for the distributive $\vee$-envelope. In proving these theorems, two different representations of $D^\wedge(L)$ will be useful, one is point-free, the other uses the set of `points' $J^\infty(L^\delta)$ of the canonical extension of $L$.

We first give the point-free construction of the distributive $\wedge$-envelope $D^\wedge(L)$ of $L$. 

To construct $D^\wedge(L)$, we want to `add joins' to $L$. This can of course be done with ideals. In the case of $D^\wedge(L)$ the required ideals will be closed under admissible joins. We thus define ``a-ideals'' as follows.

\begin{dfn}
A subset $A \subseteq L$ is called an {\it a-ideal} if (i) $A$ is a downset, i.e., if $a \in A$ and $b \leq a$ then $b \in A$, and (ii) $A$ is closed under admissible joins, i.e., if $M \subseteq A$ is join-admissible, then $\bigvee M \in A$.
\end{dfn}
This definition is a special case of a {\it $\mathcal{Z}$-join ideal} in the sense of, e.g., \cite{ErnZha2001}, and it would be interesting to see if the results in this section could be obtained using the general ideas from that line of research.

Note that any (lattice) ideal of a lattice $L$ is in particular an a-ideal. Moreover, the poset $\aidl(L)$ of all a-ideals of $L$ is a closure system: any intersection of a-ideals is again an a-ideal. Therefore, for any subset $T$ of $L$, there exists a smallest a-ideal containing $T$. We will denote this a-ideal by $\gen{T}_{ai}$ and call it {\it the a-ideal generated by $T$}. As usual, we say that an a-ideal $A$ is {\it finitely generated} if there is a finite set $T$ such that $A = \gen{T}_{ai}$. 
%\begin{rem}\label{rem:aideals}
Note that, in a distributive lattice $D$, all joins are admissible, and a-ideals coincide with lattice ideals.
%\end{rem}
The finitely generated a-ideals form a distributive lattice which will be (up to isomorphism) the distributive $\wedge$-envelope of $L$, cf. Theorem~\ref{thm:envelopeexists} below. It is possible to prove Theorem~\ref{thm:envelopeexists} directly, in a manner similar to the proof of \cite[Theorem 2]{Bru1970}. We give an alternative proof using the canonical extension and the `point-set' intuition that it provides. To this end we first show that, from the perspective of the canonical extension $L^\delta$, a set is join-admissible if, and only if, the join-irreducibles behave like join-primes with respect to the join of that set. 
\begin{lem}\label{lem:joinadmcanext}
Let $L$ be a lattice and $M \subseteq L$ a finite subset. The following are equivalent:
\begin{enumerate}
\item $M$ is join-admissible;
\item For any $x \in J^\infty(L^\delta)$, if $x \leq \bigvee M$, then $x \leq m$ for some $m \in M$.
\end{enumerate}
\end{lem}
\begin{proof}
(i) $\Rightarrow$ (ii). Suppose $M$ is join-admissible, and let $x \in J^\infty(L^\delta)$ such that $x \leq \bigvee M$. Define $x' := \bigvee_{m \in M} (x \wedge m)$. It is obvious that $x' \leq x$. We show that $x \leq x'$. Let $y$ be an ideal element of $L^\delta$ such that $x' \leq y$. Then, for each $m \in M$, we have $x \wedge m \leq y$, so by compactness, there exists $a_m \in L$ such that $x \leq a_m$ and $a_m \wedge m \leq y$. Let $a := \bigwedge_{m \in M} a_m$. Since $M$ is join-admissible, we get
\[x \leq a \wedge \bigvee M = \bigvee_{m \in M} (a \wedge m) \leq \bigvee_{m \in M} (a_m \wedge m) \leq y.\]
Since $y$ was an arbitrary ideal element above $x'$, by one of the equivalent properties of denseness (\cite[Lemma 2.4]{GehHar2001}) we conclude that $x \leq x'$. So $x = x' = \bigvee_{m \in M} (x \wedge m)$. Since $x$ is join-irreducible, we get $x = x \wedge m$ for some $m \in M$, so $x \leq m$.

(ii) $\Rightarrow$ (i). Let $a \in L$ be arbitrary. Because the other inequality is obvious, we only need to show that $a \wedge \bigvee M \leq \bigvee_{m \in M} (a \wedge m)$ holds in $L$. We show the inequality holds in $L^\delta$ and use that $L \embeds L^\delta$ is an embedding. Let $x \in J^\infty(L^\delta)$ such that $x \leq a \wedge \bigvee M$. By (ii), pick $m \in M$ such that $x \leq m$. Then $x \leq a \wedge m$, which is below $\bigvee_{m \in M} (a \wedge m)$. Since $x \in J^\infty(L^\delta)$ was arbitrary, by Proposition~\ref{prop:canextperfect} we conclude $a \wedge \bigvee M \leq \bigvee_{m \in M} (a \wedge m)$.
\end{proof}
The above lemma will be our main tool in studying admissible sets. It is a typical example of the usefulness of canonical extensions: one can formulate an algebraic property (join-admissibility) in a spatial manner (using the `points', i.e., completely join-irreducibles, of the canonical extension).

Note that the same proof goes through without the restriction that $M$ is finite, if one extends the definition of join-admissibility to include infinite sets. We will not expand on this point here, because we will only need the result for finite sets, but we merely note that this observation can be used to give an alternative proof of the results in \cite{Bru1970}.

Let us define, for any $a \in L$, $\hat{a} := \{x \in J^\infty(L^\delta) : x \leq a\}$. Lemma~\ref{lem:joinadmcanext} then says that $M$ is join-admissible if, and only if, $\hat{\bigvee M} = \bigcup_{m \in M} \hat{m}$. We can use this formulation to obtain the following characterization of the a-ideal generated by a finite subset.
\begin{lem}\label{lem:gencriterium}
Let $L$ be a lattice, $T \subseteq L$ a finite subset and $b \in L$. The following are equivalent: 
\begin{enumerate}
\item $b \in \gen{T}_{ai}$;
\item $\hat{b} \subseteq \bigcup_{a \in T} \hat{a}$;
\item There exists a finite join-admissible $M \subseteq {\downarrow}T$ such that $b = \bigvee M$.
\end{enumerate}
\end{lem}
\begin{proof}
(i) $\Rightarrow$ (ii). Note that $A := \{b \in L : \hat{b} \subseteq \bigcup_{a \in T} \hat{a}\}$ is an a-ideal which contains $T$: it is clearly a downset, and it is closed under admissible joins, using Lemma~\ref{lem:joinadmcanext}. Hence, $b \in \gen{T}_{ai} \subseteq A$, as required.

(ii) $\Rightarrow$ (iii). Let $M := \{b \wedge a \ | \ a \in T\}$. We claim that $b = \bigvee M$ and $M$ is join-admissible. Note that $\bigvee M \leq b$, so $\hat{\bigvee M} \subseteq \hat{b}$. Using (iii), we also get:
\[\hat{b} = \hat{b} \cap \bigcup_{a \in T} \hat{a} = \bigcup_{a \in T} (\hat{b} \cap \hat{a}) = \bigcup_{a \in T} \hat{b \wedge a} = \bigcup_{m \in M} \hat{m} \subseteq \hat{\bigvee M} \subseteq \hat{b}.\]
Therefore, equality holds throughout, and in particular we have that $b = \bigvee M$ and $\bigcup_{m \in M} \hat{m} = \hat{\bigvee M}$, so that $M$ is join-admissible by Lemma~\ref{lem:joinadmcanext}.

(iii) $\Rightarrow$ (i). Since $\gen{T}_{ai}$ is a downset containing $T$, $\gen{T}_{ai}$ contains $M$, and therefore, being closed under admissible joins, it contains $b = \bigvee M$.
\end{proof}
We can now give the following set-representation of the poset of finitely generated a-ideals.
\begin{prop}\label{prop:isodwedge}
Let $\phi$ be the function which sends a finitely generated a-ideal $A = \gen{T}_{ai}$ to the set $\bigcup_{a \in T} \hat{a}$. Then $\phi$ is a well-defined order isomorphism between the poset of finitely generated a-ideals and the sublattice of $\Po(J^\infty(L^\delta))$ generated by the collection $\{\hat{a} \ | \ a \in L\}$.
\end{prop}
\begin{proof}
Let $T$ and $U$ be finite subsets of $L$. Note that if $\gen{T}_{ai} = A = \gen{U}_{ai}$, then in particular $b \in \gen{T}_{ai}$ for each $b \in U$, so $\hat{b} \subseteq \bigcup_{a \in T} \hat{a}$ by Lemma~\ref{lem:gencriterium}. Hence, $\bigcup_{b \in U} \hat{b} \subseteq \bigcup_{a \in T} \hat{a}$. The proof of the other inclusion is symmetric, so indeed $\bigcup_{a \in T} \hat{a} = \bigcup_{b \in U} \hat{b}$, and $\phi$ is well defined. This argument also shows that $\phi$ is order preserving. Finally, if $\phi(\gen{U}_{ai}) = \bigcup_{b \in U} \hat{b} \subseteq \bigcup_{a \in T} \hat{a} = \phi(\gen{T}_{ai})$, then for each $b \in U$ we have $\hat{b} \subseteq \bigcup_{a \in T} \hat{a}$, so by Lemma~\ref{lem:gencriterium} we get $b \in \gen{T}_{ai}$. Since this holds for each $b \in U$, we get $\gen{U}_{ai} \subseteq \gen{T}_{ai}$, so $\phi$ is order reflecting. To see that $\phi$ is surjective, note first that, for any finite subset $T \subseteq L$, we have
$\bigcap_{a \in T} \hat{a} = \hat{\bigwedge T}.$
Hence, if $B$ is an arbitrary element of the sublattice generated by the sets $\hat{a}$, then, using distributivity in $\Po(J^\infty(L^\delta))$, we can write $B = \bigcup_{a \in T} \hat{a}$ for some finite set $T \subseteq L$. Now $B = \phi(\gen{T}_{ai})$.
\end{proof}
Note that this proposition implies in particular that the finitely generated a-ideals form a distributive lattice. In this lattice, the join of two finitely generated a-ideals is the a-ideal generated by the union of the sets of generators. The meet is simply given by intersection, as we will prove now.
\begin{prop}\label{prop:genintersection}
Let $L$ be an arbitrary lattice, and let $T$ and $U$ be finite subsets of $L$. Then
\[ \gen{T}_{ai} \cap \gen{U}_{ai} = \gen{t \wedge u \ | \ t \in T, u \in U}_{ai}.\]
In particular, the intersection of two finitely generated a-ideals is again finitely generated.
\end{prop}
\begin{proof}
By Lemma~\ref{lem:gencriterium}, we have that 
\[ \gen{T}_{ai} \cap \gen{U}_{ai} 
%= \{ b \in L \ | \ \hat{b} \subseteq \bigcup_{t \in T} \hat{t} \text{ and } \hat{b} \subseteq \bigcup_{u \in U} \hat{u}\} 
= \{ b \in L \ | \ \hat{b} \subseteq \left(\bigcup_{t \in T} \hat{t}\right) \cap \left(\bigcup_{u \in U} \hat{u}\right)\}.\]
Note that 
\[ \left(\bigcup_{t \in T} \hat{t}\right) \cap \left(\bigcup_{u \in U} \hat{u}\right) = \bigcup_{t \in T, u \in U} (\hat{t} \cap \hat{u}) = \bigcup_{t \in T, u \in U} \hat{t \wedge u}.\]
So we get that $\gen{T}_{ai} \cap \gen{U}_{ai} = \{ b \in L \ | \ \hat{b} \subseteq \bigcup_{t \in T, u \in U} \hat{t \wedge u}\} = \gen{t \wedge u \ | \ t \in T, u \in U}_{ai}$, again by Lemma~\ref{lem:gencriterium}.
\end{proof}
We are now ready to prove that the lattice of finitely generated a-ideals is indeed a distributive envelope of $L$. Let us denote by $\eta^\wedge_L$ the embedding which sends $a \in L$ to the a-ideal generated by $a$, which is simply the downset of $a$ in $L$.
%We say that a function $f : L_1 \to L_2$ between lattices {\it preserves admissible joins} if, for each finite join-admissible set $M \subseteq L_1$, we have $f(\bigvee M) = \bigvee_{m \in M} f(m)$.
%We will show that $D^\wedge(L)$ has a universal property with respect to the class of maps which preserve finite meets and admissible joins.
\begin{lem}\label{lem:emorphism}
Let $L$ be a lattice. Then $\eta^\wedge_L$ preserves finite meets and admissible joins.
\end{lem}
\begin{proof}
Clearly, $\eta^\wedge_L$ preserves finite meets. Let $M$ be a finite join-admissible set. Then $\bigvee_{m \in M} \eta^\wedge_L(m) = \gen{M}_{ai}$. By Lemma~\ref{lem:gencriterium}, we have $b \in \gen{M}_{ai}$ if, and only if, $\hat{b} \subseteq \bigcup_{m \in M} \hat{m}$, and $\bigcup_{m \in M} \hat{m} = \hat{\bigvee M}$ by Lemma~\ref{lem:joinadmcanext}, since $M$ is join-admissible. Therefore, $b \in \bigvee_{m \in M} \eta^\wedge_L(m) = \gen{M}_{ai}$ if, and only if, $\hat{b} \subseteq \hat{\bigvee M}$ if, and only if, $b \in \gen{\bigvee M}_{ai} = \eta^\wedge_L(\bigvee M)$.
\end{proof}

\begin{thm}\label{thm:envelopeexists}
Let $L$ be a lattice. The embedding $\eta^\wedge_L$ of $L$ into the finitely generated a-ideals of $L$ is a distributive $\wedge$-envelope of $L$.
\end{thm}
\begin{proof}
Let us write $D^\wedge(L)$ for the distributive lattice of finitely generated a-ideals of $L$. Lemma~\ref{lem:emorphism} shows that $\eta^\wedge_L$ preserves finite meets and admissible joins. It remains to show that it satisfies the universal property. Let $f : L \to D$ be a function which preserves meets and admissible joins. If $g : D^\wedge(L) \to D$ is a homomorphism such that $g \circ \eta^\wedge_L = f$, then, for any finite subset $T \subseteq L$, we have 
\[ g(\gen{T}_{ai}) = g\left(\bigvee_{t \in T} \eta^\wedge_L(t)\right) = \bigvee_{t \in T} g(\eta^\wedge_L(t)) = \bigvee_{t \in T} f(t).\]
So there is at most one homomorphism $g : D^\wedge(L) \to D$ satisfying $g \circ \eta^\wedge_L = f$. Let $\hat{f} : D^\wedge(L) \to D$ be the function defined for a finite subset $T \subseteq L$ by
\[ \hat{f}(\gen{T}_{ai}) := \bigvee_{t \in T} f(t).\]
We show that $\hat{f}$ is a well-defined homomorphism. For well-definedness, suppose that $\gen{T}_{ai} = \gen{U}_{ai}$ for some finite subsets $T, U \subseteq L$. Let $u \in U$ be arbitrary. We then have $u \in \gen{T}_{ai}$. By Lemma~\ref{lem:gencriterium}, $u = \bigvee M$ for some finite join-admissible $M \subseteq {\downarrow}T$. Using that $f$ preserves admissible joins and order, we get
\[ f(u) = f\left(\bigvee M\right) = \bigvee_{m \in M} f(m) \leq \bigvee_{t \in T} f(t).\]
Since $u \in U$ was arbitrary, we have shown that $\bigvee_{u \in U} f(u) \leq \bigvee_{t \in T} f(t)$. The proof of the other inequality is the same. We conclude that $\bigvee_{t \in T} f(t) = \bigvee_{u \in U} f(u)$, so $\hat{f}$ is well-defined.

It is clear that $\hat{f} \circ \eta^\wedge_L = f$. In particular, $\hat{f}$ preserves $0$ and $1$, since $f$ does. It remains to show that $\hat{f}$ preserves $\vee$ and $\wedge$. Let $T, U \subseteq L$ be finite subsets. Then $\gen{T}_{ai} \vee \gen{U}_{ai} = \gen{T \cup U}_{ai}$, so
\[\hat{f}(\gen{T}_{ai} \vee \gen{U}_{ai}) = \bigvee_{v \in T \cup U} f(v) = \bigvee_{t \in T} f(t) \vee \bigvee_{u \in U} f(u) = \hat{f}(\gen{T}_{ai}) \vee \hat{f}(\gen{U}_{ai}).\]
Using Proposition~\ref{prop:genintersection} and the assumptions that $D$ is distributive and $f$ is meet-preserving, we have
\begin{align*}
\hat{f}(\gen{T}_{ai} \wedge \gen{U}_{ai}) = \bigvee_{t \in T, u \in U} f(t \wedge u) &= \bigvee_{t \in T, u \in U} (f(t) \wedge f(u))  \\
&= \bigvee_{t \in T} f(t) \wedge \bigvee_{u \in U} f(u)  = \hat{f}(\gen{T}_{ai}) \wedge \hat{f}(\gen{U}_{ai}).\qedhere
\end{align*}
\end{proof}
We now investigate the categorical properties of the distributive $\wedge$-envelope construction a bit further. In particular, we will deduce that the assignment $L \mapsto D^\wedge(L)$ extends to an adjunction between categories. We first define the appropriate categories. We denote by $\mathbf{DL}$ the category of distributive lattices with homomorphisms. The relevant category of lattices is defined as follows.

\begin{dfn}
We say that a function $f : L_1 \to L_2$ between lattices is a ($\wedge$,a$\vee$){\it-morphism} if $f$ preserves finite meets and admissible joins, and, for any join-admissible set $M \subseteq L_1$, $f[M]$ is join-admissible. 
We denote by $\Lwav$ the category of lattices with ($\wedge$,a$\vee$)-morphisms between them. (The reader may verify that $\Lwav$ is indeed a category.) 
\end{dfn}

Note that if the lattice $L_2$ is distributive, then the condition that $f$ sends join-admissible sets to join-admissible sets is vacuously true, since any subset of a distributive lattice is join-admissible. This explains why we did not need to state the condition that $f$ preserves join-admissible sets in the universal property of $D^\wedge(L)$. However, the following example shows that in general the condition `$f$ sends join-admissible sets to join-admissible sets' can not be omitted from the definition of ($\wedge$,a$\vee$)-morphism.
\begin{exa}\label{exa:badmorphism}
{\it The composition $g \circ f$ of functions $f: L_1 \to L_2$ and $g: L_2~\to~L_3$ between lattices which preserve meets and admissible joins need not preserve admissible joins.}

Let $L_1$ be the diamond distributive lattice, let $L_2$ be the three-element antichain with $0$ and $1$ adjoined, and let $L_3$ be the Boolean algebra with $3$ atoms, as in Figure~\ref{fig:badmorphismlattices} below. Note that $L_3 = D^\wedge(L_2)$.

\vspace{-5mm}

\begin{figure}[htp]
\begin{center}
\begin{tikzpicture}[scale=0.7]
\node at (0,-1) {$L_1$};
\po{0,0}
\po{-1,1}
\po{1,1}
\po{0,2}
\li{(0,0)--(-1,1)--(0,2)--(1,1)--(0,0)}
\node at (0,0) [below] {$0$};
\node at (0,2) [above] {$1$};
\node at (-1,1) [left] {$a_1$};
\node at (1,1) [right] {$b_1$};

\node at (5,-1) {$L_2$};
\po{5,0}
\po{4,1}
\po{5,1}
\po{6,1}
\po{5,2}
\li{(5,0)--(4,1)--(5,2)--(6,1)--(5,0)}
\li{(5,2)--(5,1)--(5,0)}

\node at (5,0) [below] {$0$};
\node at (5,2) [above] {$1$};
\node at (4,1) [left] {$a_2$};
\node at (6,1) [right] {$b_2$};
\node at (5,1) [right] {$c_2$};

\node at (10,-1) {$L_3$};
\po{10,0}
\po{9,1}
\po{9,2}
\po{10,3}
\po{11,2}
\po{11,1}
\po{10,1}
\po{10,2}
\li{(10,1)--(10,0)--(9,1)--(9,2)--(10,3)--(11,2)--(11,1)--(10,0)}
\li{(10,3)--(10,2)}
\li{(9,2)--(10,1)--(11,2)}
\li{(9,1)--(10,2)--(11,1)}
\node at (10,0) [below] {$0$};
\node at (10,3) [above] {$1$};
\node at (9,1) [left] {$a_3$};
\node at (11,1) [right] {$b_3$};
\node at (10,1) [right] {$c_3$};

\end{tikzpicture} 
\end{center}

\vspace{-5mm}

\caption{The lattices $L_1$, $L_2$ and $L_3$ from Example~\ref{exa:badmorphism}.\label{fig:badmorphismlattices}}
\end{figure}
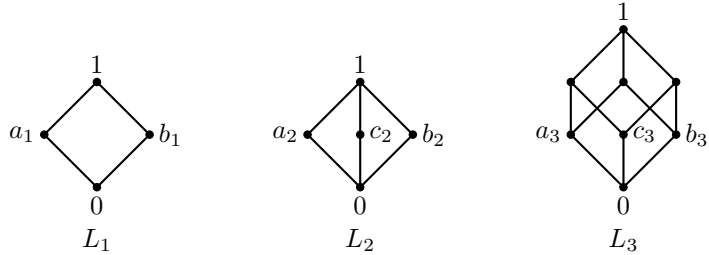

\vspace{-1em}

Let $f : L_1 \to L_2$ be the homomorphism defined by $f(x_1) = x_2$ for $x \in \{a,b\}$. Let $g : L_2 \to L_3$ be the function $\eta^\wedge_{L_2}$, i.e., $g$ is the $(\wedge,a\vee)$-morphism sending $x_2$ to $x_3$ for $x \in \{a,b,c\}$.  
%$g$ preserves admissible joins: the only non-trivial join-admissible set is $\{a_2,b_2,c_2\}$, whose join is $1$, and indeed $g(1) = g(a_2) \vee g(b_2) \vee g(c_2)$ in $L_3$. Since $L_3$ is a distributive lattice, it is clear that $g$ sends join-admissible sets to join-admissible sets.
The composition $gf$ does not preserve (admissible) joins: $gf(a_1 \vee b_1) = gf(1) = 1$, but $gf(a_1) \vee gf(b_1) = a_3 \vee b_3 \neq 1$.
Note that $f$, despite it being a homomorphism, does not send join-admissible sets to join-admissible sets: the image of $\{a_1,b_1\}$ is $\{a_2,b_2\}$, which is not join-admissible in $L_2$. 
\end{exa}
However, the following proposition shows that for {\it surjective} maps, the condition `$f$ sends join-admissible sets to join-admissible sets' can be omitted. It was already observed by Urquhart \cite{Urq1978} that surjective maps are well-behaved for duality, and accordingly our duality in Section~\ref{sec:duality} will also include all surjective lattice homomorphisms.
\begin{prop}\label{prop:surjadm}
Suppose $f : L_1 \to L_2$ is a surjective function which preserves finite meets and admissible joins. Then $f$ sends join-admissible sets to join-admissible sets {\rm (}and therefore $f$ is a morphism in $\mathbf{L_{\wedge,a\vee}}${\rm )}.
\end{prop}
\begin{proof}
Suppose $M \subseteq L_1$ is a join-admissible set. To show that $f[M]$ is join-admissible, first let $a \in L_1$ be arbitrary. Note that it follows from the definition of join-admissibility that $\{a \wedge m \ | \ m \in M\}$ is also join-admissible in $L_1$. So, using that $f$ preserves meets and admissible joins, we get
\[ f(a) \wedge \bigvee_{m \in M} f(m) = f(a \wedge \bigvee M) = f\left(\bigvee_{m \in M} (a \wedge m)\right) = \bigvee_{m \in M} (f(a) \wedge f(m)).\]
Since $f$ is surjective, any $b \in L_2$ is of the form $b = f(a)$ for some $a \in L_1$. Hence, $f[M]$ is join-admissible.
\end{proof}

Note that, if $L_1$ and $L_2$ are distributive, then ($\wedge$,a$\vee$)-morphisms from $L_1$ to $L_2$ are exactly bounded lattice homomorphisms. Hence, we have a full inclusion of categories $I^\wedge : \mathbf{DL} \embeds \Lwav$. The following is now a consequence of the universal property.
\begin{cor}\label{cor:adjunction}
The functor $D^\wedge : \Lwav \to \mathbf{DL}$, which sends $L$ to $D^\wedge(L)$ and a ($\wedge$,a$\vee$)-morphism $f : L_1 \to L_2$ to the unique homomorphic extension of the function $\eta^\wedge_{L_2} \circ f : L_1 \to D^\wedge(L_2)$, is left adjoint to $I^\wedge : \mathbf{DL} \embeds \Lwav$ and $\eta^\wedge$ is the unit of the adjunction. Moreover, the counit $\epsilon^\wedge : D^\wedge \circ I \to 1_{\mathbf{DL}}$ is an isomorphism. 
\end{cor}

\begin{comment}
(Also follows from abstract nonsense)
\begin{proof}
Note first that $D^\wedge$ is well-defined: if $f : L_1 \to L_2$ is a ($\wedge$,a$\vee$)-morphism, then $\eta^\wedge_{L_2} \circ f : L_1 \to D^\wedge(L_2)$ clearly preserves meets, and it preserves admissible joins since $f$ sends join-admissible sets to join-admissible sets and $\eta^\wedge_{L_2}$ preserves admissible joins, by Lemma~\ref{lem:emorphism}. Note that it is immediate from this definition that $D^\wedge(f) \circ \eta^\wedge_{L_1} = \eta^\wedge_{L_2} \circ f$. Now functoriality of $D^\wedge$ follows from uniqueness of extensions: if $f : L_1 \to L_2$ and $f' : L_2 \to L_3$ are ($\wedge$,a$\vee$)-morphisms, then both $D^\wedge(f') \circ D^\wedge(f)$ and $D^\wedge(f' \circ f)$ are homomorphic extensions of $\eta^\wedge_{L_3} \circ f' \circ f : L_1 \to D^\wedge(L_3)$, so they are equal by the uniqueness part of the universal property. So $D^\wedge$ is a functor and $\eta : 1_{\Lwav} \to \mathbf{DL}$ is a natural transformation. To define the counit, note that if $D$ is a distributive lattice, then, by Remark~\ref{rem:aideals}, $D^\wedge(D)$ is the poset of finitely generated lattice ideals of $D$, which is naturally isomorphic to $D$.
\end{proof}
\end{comment}

As a last general consideration about the distributive $\wedge$-envelope, we give a characterization which is closer in spirit to the one given for the infinitary version by Bruns and Lakser \cite[Corollary 2]{Bru1970}. To do so, it will be useful to know that the extension of an injective map to $D^\wedge(L)$ is injective.
\begin{prop}\label{prop:liftinjective}
Let $L$ be a lattice, $D$ a distributive lattice, and $f : L \to D$ a function which preserves finite meets and admissible joins. If $f$ is injective, then the unique extension $\hat{f} : D^\wedge(L) \to D$ is injective.
\end{prop}
\begin{proof}
Note that $f$ is order reflecting, since $f$ is meet-preserving and injective.
Suppose that $\hat{f}(\gen{U}_{ai}) \leq \hat{f}(\gen{T}_{ai})$. We need to show that $\gen{U}_{ai} \subseteq \gen{T}_{ai}$. Let $u \in U$ be arbitrary. Then $f(u) \leq \hat{f}(\gen{U}_{ai}) \leq \hat{f}(\gen{T}_{ai}) = \bigvee_{t \in T} f(t)$. For any $a \in L$, we then have
\begin{align*} 
f(a \wedge u) = f(a \wedge u) \wedge \bigvee_{t \in T} f(t) &= \bigvee_{t \in T} (f(a \wedge u) \wedge f(t)) \\
&= \bigvee_{t \in T} f(a \wedge u \wedge t) \leq f\left(\bigvee_{t \in T} (a \wedge u \wedge t)\right).
\end{align*}
Since $f$ is order reflecting, we thus get $a \wedge u \leq \bigvee_{t \in T} (a \wedge u \wedge t)$. Since the other inequality is clear, we get 
\[
a \wedge u = \bigvee_{t \in T} (a \wedge u \wedge t).
\]
In particular, putting $a = 1$, we see that $u = \bigvee_{t \in T} (u \wedge t)$, and the above equation then says that $\{u \wedge t \ | \ t \in T\}$ is join-admissible. So $u \in \gen{T}_{ai}$. We conclude that $U \subseteq \gen{T}_{ai}$, and therefore $\gen{U}_{ai} \subseteq \gen{T}_{ai}$.
\end{proof}

\begin{comment}
(old proof, using canonical extensions)
\begin{proof}
Suppose that $\hat{f}(\gen{U}_{ai}) \leq \hat{f}(\gen{T}_{ai})$. We need to show that $\gen{U}_{ai} \subseteq \gen{T}_{ai}$. Let $u \in U$ be arbitrary. Then $f(u) \leq \hat{f}(\gen{T}_{ai}) = \bigvee_{t \in T} f(t)$. We claim that $\hat{u} \subseteq \bigcup_{t \in T} \hat{t}$. From this, $u \in \gen{T}_{ai}$ follows by Lemma~\ref{lem:gencriterium}, and hence $\gen{U}_{ai} \subseteq \gen{T}_{ai}$ since $u \in U$ was arbitrary.

Let $x \in J^\infty(L^\delta)$ be arbitrary such that $x \leq u$. Write $g : L^\delta \to D^\delta$ for the canonical extension of $f$ (note that $f^\sigma = f^\pi$ since $f$ is meet-preserving), and let $g^\flat : D^\delta \to L^\delta$ be the lower adjoint of $g$. By Lemma~\ref{lem:Jimage}, $x = g^\flat(p)$ for some $p \in J^\infty(D^\delta)$. From $g^\flat(p) = x \leq u$ we get that $p \leq g(u) = f(u)$, since $g$ extends $f$. So $p \leq \bigvee_{t \in T} f(t)$, so $p \leq f(t)$ for some $t \in T$, since $D$ is distributive. We conclude that $x = g^\flat(p) \leq t$, as required.
\end{proof}
\end{comment}

The following characterisation of $D^\wedge(L)$ now follows easily.
\begin{cor}\label{cor:univchar} 
Let $L$ be a lattice. If $D$ is a distributive lattice and $f : L \to D$ is a function such that
\begin{enumerate}
\item $f$ preserves meets and admissible joins, 
\item $f$ is injective,
\item $f[L]$ is join-dense in $D$, 
\end{enumerate}
then $D$ is isomorphic to $D^\wedge(L)$ via the isomorphism $\hat{f}$.
\end{cor}
\begin{proof}
The homomorphism $\hat{f}$ is injective by Proposition~\ref{prop:liftinjective}. It is surjective because $f[L]$ is join-dense in $D$ and $\hat{f}[D^\wedge(L)] = \{\bigvee f[T] \ | \ T \subseteq L\}$, by the construction of $\hat{f}$ in the proof of Theorem~\ref{thm:envelopeexists}.
\end{proof}
%The fact that the image of the unit $\eta^\wedge$ join-generates the codomain means that $\eta^\wedge$ is an {\it epimorphism}. This corresponds to the left adjoint $D^\wedge$ being a full functor. That is, any distributive lattice homomorphism $h : D^\wedge(L) \to D^\wedge(M)$ is the image under $D^\wedge$ of some $\Lwav$-morphism (namely, of $h|_L$).

\begin{rem}\label{rem:comparebrunslakser}We compare our results in this section so far to those of Bruns and Lakser \cite{Bru1970}. The equivalence of (i) and (ii) in Lemma~\ref{lem:gencriterium} is very similar to the statement of Lemma 3 in \cite{Bru1970}. Our proofs are different from those in \cite{Bru1970} in making use of the canonical extension of $L$; in particular Lemma~\ref{lem:joinadmcanext} has proven to be useful here. Our Corollary~\ref{cor:univchar} is a finitary version of the characterisation in Corollary 2 of \cite{Bru1970}. The fact that $D^\wedge$ is an adjoint to a full inclusion can also be seen as a finitary analogue of the result of \cite{Bru1970} that their construction provides the injective hull of a meet-semilattice. Note that our construction of $D^\wedge(L)$ could also be applied to the situation where $L$ is only a meet-semilattice, if we modify our definition of join-admissible sets to require that the relevant joins exist in $L$. The injective hull of $L$ that was constructed in \cite{Bru1970} can now be retrieved from our construction by taking the free directedly complete poset (dcpo) over the distributive lattice $D^\wedge(L)$. 
This is a special case of a general phenomenon, where frame constructions may be seen as a combination of a finitary construction, followed by a dcpo construction \cite{Jung2008}.
\end{rem}

\begin{rem}\label{rem:orderdual}
We outline the order-dual version of the construction given above for later reference. A finite subset $M \subseteq L$ is {\it meet-admissible} if for all $a \in M$, $a \vee \bigwedge M = \bigwedge_{m \in M} (a \vee m)$. The universal property of the distributive $\vee$-envelope is defined as in Definition~\ref{def:distenv}, interchanging the words `join' and `meet' everywhere in the definition. An {\it a-filter} is an upset which is closed under admissible meets. The distributive $\vee$-envelope can be realized as the poset of finitely generated a-filters of $L$, ordered by reverse inclusion. The distributive $\vee$-envelope is also anti-isomorphic to the sublattice of $\Po(M^\infty(L^\delta))$ that is generated by the sets $$\check{a} := \{y \in M^\infty(L^\delta) \ | \ a \leq y\},$$ by sending the a-filter generated by a finite set $T$ to $\bigcup_{a \in T} \check{a}$. Note that the order on a-filters has to be taken to be the reverse inclusion order, to ensure that the unit embedding $\eta^\vee_L$ of the adjunction will be order-preserving. On the other hand, the order in $\Po(M^\infty(L^\delta))$ is the inclusion order, which explains why $D^\vee(L)$ is {\it anti}-isomorphic to a sublattice of $\Po(M^\infty(L^\delta))$. We say that $f : L_1 \to L_2$ is a ($\vee$,a$\wedge$){\it-morphism} if it preserves finite joins, admissible meets, and sends meet-admissible sets to meet-admissible sets. Then $D^\vee$ is a functor from the category $\Lvaw$ to $\mathbf{DL}$ which is left adjoint to the full inclusion $I^\vee : \mathbf{DL} \to \Lvaw$. We denote the unit of the adjunction by $\eta^\vee : 1_{\Lvaw} \to I^\vee \circ D^\vee$. Finally, $D^\vee(L)$ is the (up to isomorphism) unique distributive meet-dense extension of $L$ which preserves finite joins and admissible meets.
\end{rem}

%SVG
We end this section by examining additional structure which links the two distributive envelopes $D^\wedge(L)$ and $D^\vee(L)$, and enables us to retrieve $L$ from the lattices $D^\wedge(L)$ and $D^\vee(L)$. Recall from \cite[Section V.7]{Bir1967} that a tuple $(X,Y,R)$, where $X$ and $Y$ are sets and $R \subseteq X \times Y$ is a relation, is called a {\it polarity} and naturally induces a Galois connection\footnote{
%FOOTNOTE
As in \cite{Bir1967}, we use the term {\it Galois connection} for what is sometimes called a {\it contravariant adjunction}, i.e., a pair of order-preserving functions $u : P \leftrightarrows Q : l$ between posets satisfying $\id \leq ul$ and $\id \leq lu$. We reserve the word {\it adjunction} for covariant adjunctions. Both for a Galois connection and for an adjunction, we use the adjective {\it Galois-closed} for the elements $p \in P$ such that $lu(p) = p$, and for the $q \in Q$ such that $ul(q) = q$.}
 $u : \Po(X) \leftrightarrows \Po(Y) : l$.
Let $u_L : \Po(J^\infty(L^\delta)) \leftrightarrows \Po(M^\infty(L^\delta)) : l_L$ be the Galois connection associated to the polarity $(J^\infty(L^\delta),M^\infty(L^\delta),\leq_{L^\delta})$, that is, 
$$u_L(V) := \{y \in M^\infty(L^\delta) \ | \ \forall x \in V : x \leq_{L^\delta} y\} \quad \quad (V \subseteq J^\infty(L^\delta)),$$
$$l_L(W) := \{x \in J^\infty(L^\delta) \ | \ \forall y \in W : x \leq_{L^\delta} y\} \quad \quad (W \subseteq M^\infty(L^\delta)).$$ 
Note that if $V = \hat{a}$ for some $a \in L$, then $u_L(V) = u_L(\hat{a}) = \check{a}$. Recall from Proposition~\ref{prop:isodwedge} that the distributive lattice $D^\wedge(L)$ can be regarded as a sublattice of $\Po(J^\infty(L^\delta))$, and, by Remark~\ref{rem:orderdual}, $D^\vee(L)^\op$ can be regarded as a sublattice of $\Po(M^\infty(L^\delta))$. We then have:
\begin{prop}\label{prop:adjunctionD}
For any lattice $L$, the maps $u_L$ and $l_L$ restrict to a Galois connection $u_L : D^\wedge(L) \leftrightarrows D^\vee(L)^\op : l_L$. The lattice of Galois-closed elements of this Galois connection is isomorphic to $L$.
\end{prop}
{\bf N.B.} {\it The restricted Galois connection in this proposition is between $D^\wedge(L)$ and the {\rm order dual} of $D^\vee(L)$. Therefore, it is also a (covariant) adjunction between $D^\wedge(L)$ and $D^\vee(L)$.}
\begin{proof}
Note that $D^\wedge(L)$ consists of finite unions of sets of the form $\hat{a}$. If $T \subseteq L$, then we have
\[ u_L\left(\bigcup_{a \in T} \hat{a}\right) = \bigcap_{a \in T} \check{a} = \check{t},\]
where $t := \bigvee T$. From this, it follows that $u_L(D^\wedge(L)) \subseteq D^\vee(L)$, and the analogous statement for $l_L$ is proved similarly. The lattice of Galois-closed elements under this adjunction is both isomorphic to the image of $u_L$ in $D^\vee(L)$ and the image of $l_L$ in $D^\wedge(L)$. Both of these lattices are clearly isomorphic to $L$.
\end{proof}

In the presentation of $D^\wedge(L)$ and $D^\vee(L)$ as finitely generated a-ideals and a-filters, the maps $u_L$ and $l_L$ act as follows. Given an a-ideal $I$ which is generated by a finite set $T \subseteq L$, $u_L(I)$ is the principal a-filter generated by $\bigvee T$. Conversely, given an a-filter $F$ which is generated by a finite set $S \subseteq L$, $l_L(F)$ is the principal a-ideal generated by $\bigwedge S$.

In light of Proposition~\ref{prop:adjunctionD}, we can combine $D^\wedge$ and $D^\vee$ to obtain a single functor $D$ into a category of {\it adjoint pairs between distributive lattices}. On {\it objects}, this functor $D$ sends a lattice $L$ to the adjoint pair $u_L : D^\wedge(L) \leftrightarrows D^\vee(L) : l_L$ (see Proposition~\ref{prop:adjunctionD} above). For the {\it morphisms} in the domain category of $D$, we take the intersection of the set of morphisms in $\Lwav$ and the set of morphisms in $\Lvaw$. This intersection is defined directly in the following definition.
\begin{dfn}\label{dfn:admmorph}
A function $f : L \to M$ between lattices is an {\it admissible homomorphism} if it is a lattice homomorphism which sends join-admissible subsets of $L$ to join-admissible subsets of $M$ and meet-admissible subsets of $L$ to meet-admissible subsets of $M$. We denote by $\mathbf{L_a}$ the category of lattices with admissible homomorphisms.
\end{dfn}
Indeed, $f$ is an admissible homomorphism if and only if it is a morphism both in $\mathbf{L_{\wedge,a\vee}}$ and in $\mathbf{L_{\vee,a\wedge}}$. Any homomorphism whose codomain is a distributive lattice is admissible. Also, any surjective homomorphism between arbitrary lattices is admissible, by Proposition~\ref{prop:surjadm}. This may be the underlying reason for the fact that both surjective homomorphisms and morphisms whose codomain is distributive have proven to be `easier' cases in the existing literature on lattice duality (see, e.g., \cite{Urq1978, Hartung1992}). Of course, not all homomorphisms are admissible, cf. Example~\ref{exa:badmorphism} above. In the next section, we will develop a topological duality for the category $\mathbf{L_a}$.

%SVG
Let us end this section with a historical remark. The first construction of a canonical extension for lattices (although lacking an abstract characterization) was given in \cite{Harding1998}. This construction depended on the fact that any lattice occurs as the Galois-closed sets of some Galois connection. In this section we have given a `canonical' choice for this Galois connection. We will leave the precise statement of this last sentence to future work; also see the concluding section of this paper.

% SECTION ON TOPOLOGICAL DUALITY
\section{Topological duality}\label{sec:duality}
In this section, we show how the results of this paper can be applied to the topological representation theory of lattices. First, we will focus our discussion on how the existing topological dualities for lattices by Urquhart \cite{Urq1978} and Hartung \cite{Hartung1992} relate to canonical extensions. We subsequently exploit this perspective on Hartung's duality to obtain examples of lattices for which the dual space is not sober, or does not have a spectral soberification (Examples~\ref{exa:notsober}~and~\ref{exa:notspectral}, respectively). The rest of the section will be devoted to obtaining an alternative topological duality for the category of lattices $\mathbf{L_a}$. In our duality, the spaces occurring in the dual category will be spectral. 

As already remarked in \cite[Remark 2.10]{GehHar2001}, the canonical extension can be used to obtain the topological polarity in Hartung's duality for lattices \cite{Hartung1992}. We now briefly recall how this works. As is proved in \cite[Lemma 3.4]{GehHar2001}, the set $J^\infty(L^\delta)$ is in a natural bijection with the set of filters $F$ which are maximally disjoint from some ideal $I$, and the set $M^\infty(L^\delta)$ is in a natural bijection with the set of ideals which are maximally disjoint from some filter $F$. These are exactly the sets used by Hartung \cite{Hartung1992} in his topological representation for lattices. The topologies defined in \cite{Hartung1992} can be recovered from the embedding $L \embeds L^\delta$, as follows (cf. Figure~\ref{fig:egg2}). 

\vspace{-1em}

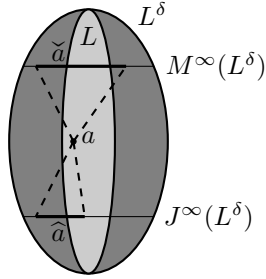
\begin{figure}[htp]
\begin{center}
\begin{tikzpicture}[scale=0.5]
\filldraw[fill=black!50!white, draw=black, thick] (0,0) ellipse (60pt and 100pt);
\node at (1.7,3.4) {$L^\delta$};

\filldraw[fill=black!20!white, draw=black, thick] (0,0) ellipse (20pt and 100pt);
\node at (0,2.8) {$L$};

\draw (-1.75,-2) -- (1.75,-2);
\node at (3.2,-2) {$J^\infty(L^\delta)$};
\draw (-1.75,2) -- (1.75,2);
\node at (3.4,2) {$M^\infty(L^\delta)$};

\po{-0.4,0}
\node at (0,0.1) {$a$};

\draw[very thick] (-1.4,-2) -- (-0.1,-2);
\draw[thick, dashed] (-0.4,0) -- (-1.4,-2);
\draw[thick, dashed] (-0.4,0) -- (-0.1,-2);
\node at (-0.8,-2.4) {$\hat{a}$};

\draw[very thick] (-1.4,2) -- (1,2);
\draw[thick, dashed] (-0.4,0) -- (-1.4,2);
\draw[thick, dashed] (-0.4,0) -- (1,2);
\node at (-0.8,2.4) {$\check{a}$};
\end{tikzpicture}
\end{center}
\caption{Topological spaces from the embedding of a lattice $L$ into its canonical extension $L^\delta$. \label{fig:egg2}}
\end{figure}

\vspace{-1mm}

For $a \in L$ we define $\hat{a} := {\downarrow}a \cap J^\infty(L^\delta)$ and $\check{a} := {\uparrow}a  \cap M^\infty(L^\delta)$. Let $\tau_c^J$ be the topology on $J^\infty(L^\delta)$ given by taking $\{\hat{a} : a \in L\}$ as a subbasis for the closed sets. Let $\tau_c^M$ be the topology on $M^\infty(L^\delta)$ given by taking $\{\check{a} : a \in L\}$ as a subbasis for the closed sets. Finally, let $R_L$ be the relation defined by $x \, R_L \, y$ if, and only if, $x \leq_{L^\delta} y$. This topological polarity $((J^\infty(L^\delta),\tau_c^J),(M^\infty(L^\delta),\tau_c^M),R_L)$ is now exactly (isomorphic to) Hartung's topological polarity $\mathbb{K}^\tau(L)$ from \cite[Definition 2.1.6]{Hartung1992}.
 %It may come as a surprise that the sets $\hat{a}$ are taken to be {\it closed}, whereas these sets are the compact {\it open} sets in the usual topology on the Stone spectrum of a distributive lattice. However, we will see in Example~\ref{exa:notsober}, among other things, that the topology on $J^\infty(L^\delta)$ obtained by taking the sets $\hat{a}$ to be open may fail to be compact when $L$ is a non-distributive lattice.

Before Hartung, Urquhart \cite{Urq1978} had already defined the dual structure of a lattice to be a doubly ordered topological space $(Z,\tau,\leq_1,\leq_2)$ whose points are maximal pairs $(F,I)$. We briefly outline how this structure can be obtained from the canonical extension. Let $P$ be the subset of $J^\infty(L^\delta) \times M^\infty(L^\delta)$ consisting of pairs $(x,y)$ such that $x \nleq_{L^\delta} y$, i.e., $P$ is the set-theoretic complement of the relation $R_L$ in Hartung's polarity. Then $P$ inherits the subspace topology from the product topology $\tau_c^J \times \tau_c^M$ on $J^\infty(L^\delta) \times M^\infty(L^\delta)$. We define an order $\preceq$ on $P$ by $(x,y) \preceq (x',y')$ iff $x \geq_{L^\delta} x'$ and $y \leq_{L^\delta} y'$; in other words, $\preceq$ is the restriction of the product of the dual order and the usual order of $L^\delta$. Urquhart's space $(Z,\tau)$ then corresponds to the subspace of $\preceq$-maximal points of $P$, and the orders $\leq_1$ and $\leq_2$ correspond to the projections of the order $\preceq$ onto the first and second coordinate, respectively.

In the following two examples, we prove that the spaces which occur in Hartung's duality may lack the nice properties that dual spaces of distributive lattices always have. 

\begin{exa}[A lattice whose dual topology is not sober]\label{exa:notsober}
Let $L$ be a countable antichain with top and bottom, as depicted in Figure~\ref{fig:Minfty}.

%\vspace{-3mm}

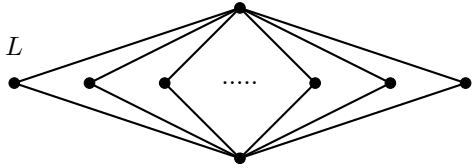
\begin{figure}[htp]
\begin{center}
\begin{tikzpicture}
\node at (0,0.5) {$L$};
\po{0,0}
\po{1,0}
\po{2,0}
\po{4,0}
\po{5,0}
\po{6,0}
\po{3,1}
\po{3,-1}

\node at (3,0) {.....};

\li{(3,-1)--(0,0)--(3,1)}
\li{(3,-1)--(1,0)--(3,1)}
\li{(3,-1)--(2,0)--(3,1)}
\li{(3,-1)--(4,0)--(3,1)}
\li{(3,-1)--(5,0)--(3,1)}
\li{(3,-1)--(6,0)--(3,1)}
\end{tikzpicture} 
\end{center}
\vspace{-1em}
\caption{The lattice $L$, a countable antichain with top and bottom.\label{fig:Minfty}}
\end{figure}

\vspace{-3mm}

One may easily show that $\id: L \to L$ is a canonical extension, so $L = L^\delta$. The set $J^\infty(L)$ is the countable antichain (as is the set $M^\infty(L)$). The topology $\tau_c^J$ is the cofinite topology on a countable set, which is not sober: the entire space is itself a closed irreducible subset which is not the closure of a point.

Also note that if one instead would define a topology on $J^\infty(L)$ by taking the sets $\hat{a}$, for $a \in L$, to be open, instead of closed, then one obtains the discrete topology on $J^\infty(L)$, which is in particular not compact.\qed
\end{exa}

In the light of the above example, one may wonder whether the {\it soberification} of the space $(J^\infty(L),\tau_c^J)$ might have better properties, and in particular whether it will be spectral. However, the following example shows that it cannot be, since the frame of opens of the topological space $(J^\infty(L),\tau_c^J)$ in the following example fails to have the property that intersections of compact elements are compact.

\begin{exa}[A lattice whose dual topology is not arithmetic]\label{exa:notspectral}
Consider the lattice $K$ depicted in Figure~\ref{fig:notspectral}. In this figure, the elements of the original lattice $K$ are drawn as filled dots, and the three additional elements $a$, $b$ and $c$ of the canonical extension $K^\delta$ are drawn as unfilled dots. 

%\vspace{-5mm}

%\vspace{-6mm}

The set $J^\infty(K^\delta)$ is $\{b_i, c_i, z_i \ | \ i \geq 0\} \cup \{b,c\}$. We take $\{\hat{a} : a \in K\}$ as a subbasis for the closed sets, so $\{(\hat{a})^c : a \in K\}$ is a subbasis for the open sets.

In particular, $(\hat{b_0})^c$ and $(\hat{c_0})^c$ are compact open sets. However, their intersection is not compact: $\{ (\hat{a_n})^c \}_{n=0}^\infty$ is an open cover of $(\hat{b_0})^c \cap (\hat{c_0})^c = \{z_i : i \geq 0\}$ with no finite subcover. \qed
\end{exa}
\newpage

\begin{figure}[htp]
\begin{center}
\begin{tikzpicture}[scale=0.7]
\node at (7,6) {$K \embeds K^\delta$};

\po{1,3}
\node[left] at (1,2.8) {$b_2$};

\po{3,3}
\node[right] at (3,2.8) {$c_2$};

\po{2,4}
\node[right] at (2,4.1) {$a_2$};

\li{(1,3)--(2,4)--(3,3)}

\po{1,4}
\node[left] at (1,3.8) {$b_1$};

\po{3,4}
\node[right] at (3,3.8) {$c_1$};

\po{2,5}
\node[right] at (2,5.1) {$a_1$};

\li{(1,3)--(1,4)}
\li{(3,3)--(3,4)}
\li{(2,4)--(2,5)}

\li{(1,4)--(2,5)--(3,4)}

\po{1,5}
\node[left] at (1,4.8) {$b_0$};

\po{3,5}
\node[right] at (3,4.8) {$c_0$};

\po{2,6}
\node[right] at (2,6.1) {$a_0$};

\li{(1,4)--(1,5)}
\li{(3,4)--(3,5)}
\li{(2,5)--(2,6)}

\li{(1,5)--(2,6)--(3,5)}

\node[rotate=90] at (1,2) {$\dots$};
\node[rotate=90] at (3,2) {$\dots$};
\node[rotate=90] at (2,3) {$\dots$};

\po{2,0}
\node[below] at (2,0) {$0$};

\li{(2,0)--(1,1)--(2,2)--(3,1)--(2,0)}

\filldraw[fill=white](1,1) circle (2 pt);
\node[left] at (1,1) {$b$};

\filldraw[fill=white](2,2) circle (2 pt);
\node[above] at (2,2) {$a$};

\filldraw[fill=white](3,1) circle (2 pt);
\node[right] at (3,1) {$c$};

\po{7.5,1}
\node[right] at (7.5,1.2) {$z_0$};
\po{6.5,1}
\node[right] at (6.5,1.2) {$z_1$};
\po{5.5,1}
\node[right] at (5.5,1.2) {$z_2$};
\node at (4.5, 1) {$\dots$};

\draw (2,6) to [bend left=10] (7.5,1);
\draw (2,5) to [bend left=10] (6.5,1);
\draw (2,4) to [bend left=10] (5.5,1);
\li{(7.5,1)--(2,0)}
\li{(6.5,1)--(2,0)}
\li{(5.5,1)--(2,0)}
\end{tikzpicture} 

%\vspace{-5mm}

\caption{The lattice $K$, for which $(J^\infty(K^\delta),\tau_c)$ is not spectral.\label{fig:notspectral}}
\end{center}
\end{figure}
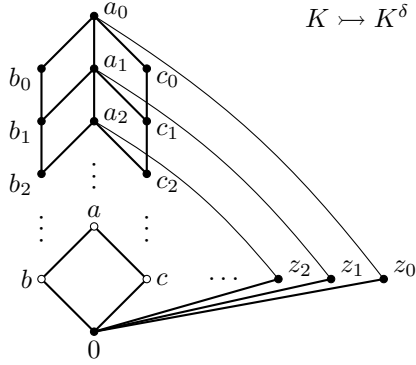

\vspace{-3em}

The above examples indicate that the spaces obtained from Hartung's duality can be badly behaved. In particular, they do not fit into the framework of the duality between sober spaces and spatial frames. The individual spaces which occur in Hartung's duality may fail to be the Stone duals of any distributive lattice.

In the remainder of this section, we combine the facts from Section~\ref{sec:envelopes} with the existing Stone-Priestley duality for distributive lattices to obtain a duality for a category of lattices with admissible homomorphisms (see Definition~\ref{dfn:admmorph} below). Since the $\Lwav$-morphisms are {\it exactly} the morphisms which can be lifted to homomorphisms between the $D^\wedge$-envelopes, these morphisms also correspond exactly to the lattice morphisms which have functional duals between the $X$-sets of the dual polarities; the same remark applies to $\Lvaw$-morphisms and the $Y$-sets of the dual polarities. The duals of admissible morphisms will be pairs of functions; one function being the dual of the `$\Lwav$-part' of the admissible morphism, the other being the dual of the `$\Lvaw$-part' of the morphism.

We will now first define an auxiliary category of `doubly dense adjoint pairs between distributive lattices' ($\mathbf{daDL}$) which has the following two features:
\begin{enumerate}
\item The category $\mathbf{L_a}$ can be embedded into $\mathbf{daDL}$ as a full subcategory (Proposition~\ref{prop:latticesasdaDLs});
\item There is a natural Stone-type duality for $\mathbf{daDL}$ (Theorem~\ref{thm:dualitydaDL}).
\end{enumerate}
We will then give a dual characterization of the `special' objects in $\mathbf{daDL}$ which are in the image of the embedding of $\mathbf{L_a}$ from (i), calling these dual objects {\it tight} (cf. Definition~\ref{dfn:tight}). The restriction of the natural Stone-type duality (ii) will then yield our final result: a topological duality for lattices with admissible homomorphisms (Theorem~\ref{thm:mainduality}).

\begin{dfn}\label{dfn:daDL}We denote by $\mathbf{aDL}$ the category with:
\begin{itemize}
\item {\it objects:} tuples $(D,E,f,g)$, where $D$ and $E$ are distributive lattices and $f : D \leftrightarrows E : g$ is a pair of maps such that $f$ is lower adjoint to $g$; 
\item {\it morphisms:} an $\mathbf{aDL}$-morphism from $(D_1,E_1,f_1,g_1)$ to $(D_2,E_2,f_2,g_2)$ is a pair of homomorphisms $h^\wedge : D_1 \to D_2$ and $h^\vee : E_1 \to E_2$ such that $h^\vee f_1 = f_2 h^\wedge$ and $h^\wedge g_1 = g_2 h^\vee$, i.e., both squares in the following diagram commute:
\begin{center}
\begin{tikzpicture}[scale=0.8]
\matrix (m) [matrix of math nodes, row sep=3em, column sep=2.5em, text height=1.5ex, text depth=0.25ex] 
{ D_1 & & E_1 \\
D_2 & & E_2 \\};
\path[->] (m-1-1) edge[bend right=10] node[below] {$f_1$} (m-1-3);
\path[->] (m-1-3) edge[bend right=10] node[above] {$g_1$} (m-1-1);
\path[->] (m-2-1) edge[bend right=10] node[below] {$f_2$} (m-2-3);
\path[->] (m-2-3) edge[bend right=10] node[above] {$g_2$} (m-2-1);

\path[->] (m-1-1) edge node[left] {$h^\wedge$} (m-2-1);
\path[->] (m-1-3) edge node[right] {$h^\vee$} (m-2-3);
\end{tikzpicture}
\end{center}
\end{itemize}

We call an adjoint pair $(D,E,f,g)$ {\it doubly dense} if both $g[E]$ is join-dense in $D$ and $f[D]$ is meet-dense in $E$. We denote by $\mathbf{daDL}$ the full subcategory of $\mathbf{aDL}$ whose objects are doubly dense adjoint pairs.
\end{dfn}

\begin{prop}\label{prop:latticesasdaDLs}
The category $\mathbf{L_a}$ is equivalent to a full subcategory of $\mathbf{daDL}$.
\end{prop}
\newcommand{\D}{\mathcal{D}}
\begin{proof}
Let $\D : \mathbf{L_a} \to \mathbf{daDL}$ be the functor defined by sending: 
\begin{itemize}
\item a lattice $L$ to $\D(L) := (D^\wedge(L),D^\vee(L),u,l)$,
\item an admissible morphism $h : L_1 \to L_2$ to the pair $\D(h) := (D^\wedge(h), D^\vee(h))$.
\end{itemize}
We show that $\D$ is a well-defined full and faithful functor.

For objects, note that $\D(L)$ is a doubly dense adjoint pair by Corollary~\ref{cor:univchar} and Proposition~\ref{prop:adjunctionD} in the previous section. 

Let $h : L_1 \to L_2$ be an admissible morphism. We need to show that $\D(h)$ is a morphism of $\mathbf{daDL}$, i.e., that $u_{L_2} \circ D^\wedge(h) = D^\vee(h) \circ u_{L_1}$ and $l_{L_2} \circ D^\vee(h) = D^\wedge(h) \circ l_{L_1}$. Since $D^\wedge(L_1)$ is join-generated by the image of $L_1$, and both $u_{L_2} \circ D^\wedge(h)$ and $D^\vee(h) \circ u_{L_1}$ are join-preserving, it suffices to note that the diagram commutes for elements in the image of $L_1$. This is done by the following diagram chase:
\[ u_{L_2} \circ D^\wedge(h) \circ \eta^\wedge_{L_1} = u_{L_2} \circ \eta^\wedge_{L_2} \circ h = \eta^\vee_{L_2} \circ h = D^\vee(h) \circ \eta^\vee_{L_1}= D^\vee(h) \circ u_{L_1} \circ \eta^\wedge_{L_1},\]
where we have used that $\eta^\wedge$ is a natural transformation and that $u_L \circ \eta^\wedge_L = \eta^\vee_L$. The proof that $l_{L_2} \circ D^\vee(h) = D^\wedge(h) \circ l_{L_1}$ is similar.

It remains to show that the assignment $h \mapsto \D(h)$ is a bijection between $\mathbf{L_a}(L_1,L_2)$ and $\mathbf{daDL}(\D(L_1),\D(L_2))$. If $(h^\wedge,h^\vee) : \D(L_1) \to \D(L_2)$ is a $\mathbf{daDL}$-morphism, then $h^\wedge$ maps lattice elements to lattice elements. That is, the function $h^\wedge \circ \eta^\wedge_{L_1} : L_1 \to \D(L_2)$ maps into $\im(\eta^\wedge_{L_2}) = \im(l_{L_2})$, since
\begin{align*}
h^\wedge \circ \eta^\wedge_{L_1} &= h^\wedge \circ l_{L_1} \circ u_{L_1} \circ \eta^\wedge_{L_1} \\
&= l_{L_2} \circ h^\vee \circ u_{L_1} \circ \eta^\wedge_{L_1}.
\end{align*}
We may therefore define $h : L_1 \to L_2$ to be the function $(\eta^\wedge_{L_2})^{-1} \circ h^\wedge \circ \eta^\wedge_{L_1}$. Note that this function is equal to $(\eta^\vee_{L_2})^{-1} \circ h^\vee \circ \eta^\vee_{L_1}$, since
\begin{align*}
(\eta^\vee_{L_2})^{-1} \circ h^\vee \circ \eta^\vee_{L_1} &= (\eta^\vee_{L_2})^{-1} \circ h^\vee \circ u_{L_1} \circ l_{L_1} \circ \eta^\vee_{L_1} \\
&= (\eta^\vee_{L_2})^{-1} \circ u_{L_2} \circ h^\wedge \circ \eta^\wedge_{L_1} \\
&=  (\eta^\wedge_{L_2})^{-1} \circ h^\wedge \circ \eta^\wedge_{L_1},
\end{align*}
where we have used that, for any lattice $L$, $l_L \circ \eta^\vee_{L} = \eta^\wedge_{L}$ and $u_L \circ \eta^\wedge_{L} = \eta^\vee_L$.
So, since $(\eta^\wedge_{L_2})^{-1} \circ h^\wedge \circ \eta^\wedge_{L_1} = h = (\eta^\vee_{L_2})^{-1} \circ h^\vee \circ \eta^\vee_{L_1}$, it is clear that $h$ is a homomorphism, since the left-hand-side preserves $\wedge$ and the right-hand-side preserves $\vee$. It remains to show that $h$ is admissible, i.e., that $h$ sends join-admissible subsets to join-admissible subsets, and meet-admissible subsets to meet-admissible subsets. Note that, by the adjunction in Corollary~\ref{cor:adjunction}, if a function $k : L \to D$ admits a homomorphic extension $\hat{k} : D^\wedge(L) \to D$, then $k$ is $(\wedge,a\vee)$-preserving, since it is equal to the composite $\hat{k} \circ \eta^\wedge_L$. In particular, the map $\eta^\wedge_{L_2} \circ h$ is $(\wedge,a\vee)$-preserving, its homomorphic extension being $h^\wedge$. It follows from this that $h$ sends join-admissible subsets to join-admissible subsets, since join-admissible subsets are the {\it only} subsets whose join is preserved by $\eta^\wedge_{L_2}$. The proof that $h$ preserves meet-admissible subsets is similar.

\begin{comment}
To this end, let $M \subseteq L_1$ be a join-admissible subset. Note first that $\eta^\wedge_{L_2} \circ h = h^\wedge \circ \eta^\wedge_{L_1}$ preserves meets and admissible joins. Now, for any $b \in L_2$, note that in $D^\wedge(L_2)$, we have
\begin{align*}
\eta^\wedge_{L_2}\left(\bigvee_{m \in M} (b \wedge h(m))\right) &\geq \bigvee_{m \in M}\eta^\wedge_{L_2}(b \wedge h(m))  &(\eta^\wedge_{L_2} \text{ monotone}) \\
&= \bigvee_{m\in M} (\eta^\wedge_{L_2}(b) \wedge \eta^\wedge_{L_2}h(m)) &(\eta^\wedge_{L_2} \text{$\wedge$-preserving}) \\
&= \eta^\wedge_{L_2}(b) \wedge \bigvee_{m\in M} (\eta^\wedge_{L_2}h(m)) &(\text{distributivity}) \\
&=  \eta^\wedge_{L_2}(b) \wedge \eta^\wedge_{L_2}h\left(\bigvee M\right) &(\eta^\wedge_{L_2} \circ h \text{ a$\vee$-preserving})\\
&= \eta^\wedge_{L_2}\left(b \wedge h\left(\bigvee M\right)\right) &(\eta^\wedge_{L_2} \text{$\wedge$-preserving})\\
&=  \eta^\wedge_{L_2}\left(b \wedge \bigvee h(M)\right) &(h \text{ $\vee$-preserving})
\end{align*}
so that $\bigvee_{m \in M} (b \wedge h(m)) \geq b \wedge \bigvee h(M)$, since $\eta^\wedge_{L_2}$ is an embedding.

The proof that $h$ sends meet-admissible subsets to meet-admissible subsets is similar. 
\end{comment}

Now, since $h^\wedge \circ \eta^\wedge_{L_1} = \eta^\wedge_{L_2} \circ h$, we have that $h^\wedge = D^\wedge(h)$, since $D^\wedge(h)$ was defined as the unique homomorphic extension of $\eta^\wedge_{L_2} \circ h$, and similarly $h^\vee = D^\vee(h)$. We conclude that $(h^\wedge,h^\vee) = \D(h)$, so $h \mapsto \D(h)$ is surjective.

It is clear that if $h \neq h'$, then $D^\wedge(h) \neq D^\wedge(h')$, so $\D(h) \neq \D(h')$. Hence, the assignment $h \mapsto \D(h)$ is bijective, as required.
\end{proof}

\begin{exa}
[Not every object of $\mathbf{daDL}$ is the distributive envelope of a lattice] 
\label{exa:notalldaDLs}
Take any distributive lattice $D$ and consider the daDL $(F_\vee(D,\wedge), F_\wedge(D, \vee), f,g)$, where $F_\vee(D,\wedge)$ is the free join-semilattice generated by the meet-semilattice reduct of $D$ viewed as a distributive lattice, $F_\wedge(D, \vee)$ is defined order dually, and $f$ and $g$ both are determined by sending each generator to itself. Such a daDL is not of the form we are interested in since the $\wedge$- and $\vee$-envelopes of any distributive lattice both are equal to the lattice itself since all joins are admissible.
\end{exa}

The above example shows that the category $\mathbf{L_a}$, that we will be most interested in, is a proper subcategory of $\mathbf{daDL}$, but we start by giving a description of the topological duals of the objects of $\mathbf{daDL}$.
To this end, let $(D,E,f,g)$ be a doubly dense adjoint pair. If $X$ and $Y$ are the dual Priestley spaces of $D$ and $E$ respectively, then it is well-known that an adjunction $(f,g)$ corresponds to a relation $R$ satisfying certain properties. In our current setting of {\it doubly dense} adjoint pairs, it turns out that it suffices to consider the topological reducts of the Priestley spaces $X$ and $Y$ (i.e., forgetting the order) and the relation $R$ between them. Both the Priestley orders of the spaces $X$ and $Y$ and the adjunction $(f,g)$ can be uniquely reconstructed from the relation $R$, as we will prove shortly. The dual of a doubly dense adjoint pair will be a {\it totally separated compact polarity} (TSCP), which we define to be a polarity $(X,Y,R)$, where $X$ and $Y$ are Boolean spaces and $R$ is a relation from $X$ to $Y$, satisfying certain properties (see Definition~\ref{dfn:TSCP} for the precise definition).

We now first fix some useful terminology for topological polarities, regarding the closure and interior operators induced by a polarity, its closed and open sets, and its associated quasi-orders.

Let $X$ and $Y$ be sets and $R\subseteq X\times Y$. Then we obtain a closure operator $\overline{(\ \ )}$ on $X$ given by 
\[
\overline{S} := \{x\in X\mid R[x]\subseteq R[S]\}\ \mbox{ for } S\subseteq X.
\]
The subsets $S$ of $X$ satisfying $\overline{S}=S$ we will call {\it $R$-closed}. The $R$-closed subsets of $X$ form a lattice in which the meet is intersection and join is the closure of the union. 
We of course also obtain an adjoint pair of maps:
\begin{center}
\begin{tikzpicture}
\matrix (m) [matrix of math nodes, row sep=3em, column sep=2.5em, text height=1.5ex, text depth=0.25ex] 
{ {\mathcal P}(X) & & {\mathcal P}(Y)  \\};
\path[->] (m-1-1) edge[bend right=10] node[below] {$\Diamond$} (m-1-3);
\path[->] (m-1-3) edge[bend right=10] node[above] {$\Box$} (m-1-1);
\end{tikzpicture}
\end{center}
given by 
\[
\Diamond S=R[S]=\{y\in Y\mid \exists x\in S \ xRy\}
\]
 and 
 \[
 \Box T=\left(R^{-1}[T^c]\right)^c=\{x\in X\mid \forall y\in Y ( xRy \implies y\in T )\}.
 \]
The relation with the closure operator on $X$ is that $\overline{S}=\Box\Diamond S$. Note also that on points of $X$ this yields a quasi-order given by 
\[
x'\leq x \ \iff\ R[x']\subseteq R[x].
\]
Similarly, on $Y$ we obtain an interior operator
\[
T^\circ=\{y\in Y\mid \exists x\in X \left[ xRy \mbox{ and } \forall y'\in Y (xRy'\implies y'\in T)\right]\}=\Diamond\Box T
\]
and a quasi-order on $Y$ given by 
\[
y\leq y' \ \iff\ R^{-1}[y]\supseteq R^{-1}[y'].
\]
The range of $\Diamond$ is equal to the range of the interior operator, and we call these sets $R$-open. This collection of subsets of $Y$ forms a lattice isomorphic to the one of $R$-closed subsets of $X$. In this incarnation, the join is given by union whereas the meet is given by interior of the intersection. Note that the $R$-closed subsets of $X$ as well as the $R$-open subsets of $Y$ all are down-sets in the induced quasi-orders. 

We are now ready to define the objects which will be dual to doubly dense adjoint pairs. 

\begin{dfn}\label{dfn:TSCP}
A {\it topological polarity} is a tuple $(X,Y,R)$, where $X$ and $Y$ are topological spaces and $R$ is a relation. A {\it compact polarity} is a topological polarity in which both $X$ and $Y$ are compact. A topological polarity is {\it totally separated} if it satisfies the following conditions:
\begin{enumerate}
\item {\it ($R$-separated)} The quasi-orders induced by $R$ on $X$ and $Y$ are partial orders.
\item {\it ($R$-operational)} For each clopen down-set $U$ of $X$, the image $\Diamond U$ is clopen in $Y$;  For each clopen down-set $V$ of $Y$, the image $\Box V$ is clopen in $X$;
\item {\it (Totally $R$-disconnected)} For each $x\in X$ and each $y\in Y$, if $\neg (x R y)$ then there are clopen sets $U\subseteq X$ and $V\subseteq Y$ with $\Diamond U=V$ and $\Box V=U$ so that $x\in U$, and $y\not\in V$.
\end{enumerate}
In what follows, we often abbreviate ``totally separated compact polarity'' to TSCP.
\end{dfn}

\begin{rem}
In the definition of totally separated topological polarities, the first property states that $R$ separates the points of $X$ as well as the points of $Y$. The second property states that $R$ yields operations between the clopen downsets of $X$ and of $Y$. Finally, the third property generalizes total order disconnectedness, well known from Priestley duality, hence the name total R-disconnectedness. 
\end{rem}
The following technical observation about total $R$-disconnectedness will be useful in what follows. 
\begin{lem}\label{lem:Lsep}
If a topological polarity $(X,Y,R)$ is totally $R$-disconnected, then the following hold:
\begin{itemize}
\item $\text{If } x'\nleq x \text{ then there exists }U\subseteq X \text{ clopen and $R$-closed such that } \\ x\in U \text{ and } x'\not\in U. $

\item $\text{If } y'\nleq y \text{ then there exists }V\subseteq Y \text{ clopen and $R$-open such that } \\ y\in V \text{ and } y'\not\in V. $
\end{itemize}
\end{lem}
\begin{proof}
Suppose that $x' \nleq x$. By definition of $\leq$, there exists $y \in Y$ such that $x' \, R \, y$ and $\neg(x\,R\,y)$. By total R-disconnectedness, there exist clopen $U$ and $V$ such that $\Diamond U = V$, $\Box V = U$, $x \in U$ and $y \not\in V$. We now have $x' \not\in U$, for otherwise we would get $y \in \Diamond U = V$. Since $U = \Box V = \Box \Diamond U$, we get that $U$ is $R$-closed, as required. The proof of the second property is dual.
\end{proof}

Now, given a daDL $(D,E,f,g)$, we call its {\it dual polarity} the tuple $(X,Y,R)$, where $X$ and $Y$ are the topological reducts of the Priestley dual spaces of $D$ and $E$, respectively (which are in particular compact), and $R$ is the relation defined by
\[ 
x \, R \, y \iff f[x] \subseteq y,
\]
where we regard the points of $X$ and $Y$ as prime filters of $D$ and $E$, respectively.

Conversely, given a totally separated compact polarity $(X,Y,R)$, we call its {\it dual adjoint pair} the tuple $(D,E,\Diamond, \Box)$, where $D$ and $E$ are the lattices of clopen downsets of $X$ and $Y$ in the induced orders, respectively, and $\Diamond$ and $\Box$ are the operations defined above (note that these operations are indeed well-defined by item (ii) in the definition of totally separated).

Note that if $L$ is a distributive lattice, then its associated daDL is $(L,L,\id,\id)$, which has dual polarity $(X,X,\leq)$, where $(X,\leq)$ is the usual Priestley dual space of $L$. Thus, the above definitions generalize Priestley duality.

The following three propositions constitute the object part of our duality for doubly dense adjoint pairs.
\begin{prop}
If $(D,E,f,g)$ is a doubly dense adjoint pair, then its dual polarity $(X,Y,R)$ is compact and totally separated.
\end{prop}

\begin{proof}
Let $(D,E,f,g)$ be a doubly dense adjoint pair, and let $L$ be the lattice which is isomorphic to both the image of $g$ in $D$ and to the image of $f$ in $E$.

The dual polarity $(X,Y,R)$ is compact because the dual Priestley spaces of $D$ and $E$ are compact.

For $R$-separation, suppose that $x \neq x'$ in $X$. We need to show that $R[x] \neq R[x']$. Without loss of generality, pick $d \in D$ such that $d \in x$ and $d \not\in x'$. Since $L$ is join-dense in $D$ and $x$ is a prime filter, there exists $a \in L$ with $a \leq d$, such that $a \in x$. Note that $a \not\in x'$ since $a \leq d$ and $d \not\in x'$. It follows that $f(a) \not\in f[x']$: if we would have $d' \in x'$ such that $f(a) = f(d')$, then we would get $d' \leq gf(d') = gf(a) = a$, contradicting that $a \not\in x'$. By the prime filter theorem, there exists a prime filter $y \subseteq E$ such that $f[x'] \subseteq y$ and $f(a) \not\in y$. Since we do have $f(a) \in f[x]$, it follows that $x' \, R \, y$ and $\neg(x \, R \, y)$, so $R[x'] \neq R[x]$, as required. The proof that $R$ induces a partial order on $Y$ is similar.

%Since $x' \nleq x$ in $D^\delta$, there exists $d \in D$ such that $x \leq d$ and $x' \nleq d$. Since $L$ is join-dense in $D$ and $x$ is completely join-prime in $D^\delta$, there exists $a \in \im(g)$ with $a \leq d$, such that $x \leq a$ and still $x' \nleq a$. We then get that $x \leq gf(a) = a$, so $f^\delta(x) \leq f(a)$, while $f^\delta(x') \nleq f(a)$ for similar reasons. In particular, $f^\delta(x') \nleq f^\delta(x)$, so there also exists $y \leq f^\delta(x')$ such that $y \nleq f^\delta(x)$, which says precisely that $x' \, R \, y$ and $\neg(x \, R \, y)$, so $R[x'] \not\subseteq R[x]$, as required. The proof that $R$ induces a partial order on $Y$ is similar.

For $R$-operationality, it suffices to observe that, for any $d \in D$, we have $\Diamond \hat{d} = R[\hat{d}] = \hat{f(d)}$ and, for any $e \in E$, we have $\Box \hat{e} = \hat{g(e)}$. 
%: if $y \in R[\hat{d}]$ then $xRy$ for some $x \leq d$ and we get $y \leq f^\delta(x) \leq f(d)$. Conversely, if $y \leq f(d)$, then $d \nleq g(\kappa(y))$, and we find $x \leq d$ such that $x \nleq g(\kappa(y))$, i.e., $xRy$. 

For total $R$-disconnectedness, suppose that $\neg (x \, R \, y)$. This means that $f[x] \not\subseteq y$, so there is $d \in D$ such that $d \in x$ and $f(d) \not\in y$. Since $d \leq gf(d)$, we get $gf(d) \in x$, so we may put $U := \hat{gf(d)}$ and $V := \hat{f(d)}$.
\end{proof}

\begin{prop}
If $(X,Y,R)$ is a totally separated compact polarity, then its dual adjoint pair is doubly dense.
\end{prop}
\begin{proof}
From what was stated in the preliminaries above, it is clear that we get an adjoint pair between the lattices of clopen downsets. We need to show that it is doubly dense. 

To this end, let $U$ be a clopen downset of $X$. We show that $U$ is a finite union of clopen $R$-closed sets. First fix $x \in U$. For any $x' \not\in U$, we have that $x' \nleq x$. By $L$-separation, pick a clopen $R$-closed set $U_{x'}$ such that $x \in U_{x'}$ and $x' \not\in U_{x'}$. Doing this for all $x' \not\in U$, we obtain a cover $\{U_{x'}^c\}_{x' \not\in U}$ by clopen sets of the compact set $U^c$. Therefore, there exists a finite subcover $\{U_i^c\}_{i=1}^n$ of $U^c$. Let us write $V_x := \bigcap_{i=1}^n U_i$. We then get that $x \in V_x \subseteq U$, and $V_x$ is clopen $R$-closed, since each of the $U_i$ is. Doing this for all $x \in U$, we get a cover $\{U_x\}_{x \in X}$ by clopen $R$-closed sets of the compact set $U$, which has a finite subcover. This shows that $U$ is a finite union of clopen $R$-closed sets.

The proof that clopen downsets of $Y$ are finite intersections of clopen $R$-open sets is essentially dual; we leave it to the reader.
\end{proof}

\begin{prop}\label{prop:TSCPdoubledual}
Any totally separated compact polarity is isomorphic to its double dual. 

More precisely, if $(X,Y,R)$ is a TSCP, let $(X',Y',R')$ be the dual polarity of the dual adjoint pair of $(X,Y,R)$. Then there are homeomorphisms $\phi : X \to X'$, $\psi : Y \to Y'$ such that $x \, R \, y$ iff $\phi(x) \, R' \psi(y)$.
\end{prop}
\begin{proof}
Note that if $(X,Y,R)$ is a TSCP, then $X$ and $Y$ with the induced orders are Priestley spaces: total-order-disconnectedness follows from Lemma~\ref{lem:Lsep} and the fact, noted above, that $R$-closed and $R$-open sets are downsets in the induced orders.

Therefore, by Priestley duality we have homeomorphisms $\phi : X \to X'$ and $\psi : Y \to Y'$, both given by sending points to their neighbourhood filters of clopen downsets.

It remains to show that $\phi$ and $\psi$ respect the relation $R$. Note that, by definition, we have $x' \, R' \, y'$ iff for any clopen downset $U$ in $x'$, we have that $R[U]$ is in $y'$. Suppose $x \, R \, y$, and that $U \in \phi(x)$. Then $x \in U$, so $y \in R[U]$, so $R[U] \in \psi(y)$. Conversely, suppose that $\neg (x \, R \, y)$. By total $R$-disconnectedness, we pick a clopen $R$-closed set $U$ with $x \in U$ and $y \not\in \Diamond U = R[U]$. This set $U$ is a clopen downset which witnesses that $\neg (\phi(x) \, R' \, \psi(y))$.
\end{proof}
We can extend this object correspondence between daDL's and TSCP's to a dual equivalence of categories. The appropriate morphisms in the category of totally separated compact polarities are pairs of functions $(s_X,s_Y)$, which are the Priestley duals of $(h^\wedge,h^\vee)$. The condition that morphisms in $\mathbf{daDL}$ make two squares commute (see Definition~\ref{dfn:daDL}) dualizes to back-and-forth conditions on $s_X$ and $s_Y$, as in the following definition.
\begin{dfn}
A {\it morphism} in the category $\mathbf{TSCP}$ of totally separated compact polarities from $(X_1,Y_1,R_1)$ to $(X_2,Y_2,R_2)$ is a pair $(s_X,s_Y)$ of continuous functions $s_X : X_1 \to X_2$ and $s_Y : Y_1 \to Y_2$, such that, for all $x \in X_1$, $x' \in X_2$, $y \in Y_1$, $y' \in Y_2$:
\begin{itemize}
\item[(forth)] If $x \,R_1\,y$, then $s_X(x) \,R_2 \, s_Y(y)$,
\item[($\Diamond$-back)] If $x' \, R_2 \, s_Y(y)$, then there exists $z \in X_1$ such that $z \, R_1 \, y$ and $s_X(z) \leq x'$,
\item[($\Box$-back)] If $s_X(x) \,R_2 \, y'$, then there exists $w \in Y_1$ such that $x \, R_1 \, w$ and $y' \leq s_Y(w)$.
\end{itemize}
\end{dfn}
The conditions on these morphisms should look natural to readers who are familiar with back-and-forth conditions in modal logic. More detailed background on how these conditions arise naturally from the theory of canonical extensions can be found in \cite[Section 5]{Geh2012}.

\begin{thm}\label{thm:dualitydaDL}
The category $\mathbf{daDL}$ is dually equivalent to the category $\mathbf{TSCP}$.
\end{thm}
\begin{proof}
The hardest part of this theorem is the essential surjectivity of the functor which assigns to a daDL its dual polarity. We proved this in Proposition~\ref{prop:TSCPdoubledual}.  One may then either check directly that the assignment which sends a $\mathbf{daDL}$-morphism $(h^\wedge,h^\vee)$ to the pair $(s_X,s_Y)$ of Priestley dual functions between the spaces in the dual polarities is a bijection between the respective $\mathbf{Hom}$-sets, or refer to \cite[Section 5]{Geh2012} for a more conceptual proof using canonical extensions.
\end{proof}

In particular, combining Theorem~\ref{thm:dualitydaDL} with Proposition~\ref{prop:latticesasdaDLs}, the category $\mathbf{L_a}$ of lattices with admissible homomorphisms is dually equivalent to a full subcategory of $\mathbf{TSCP}$. The task that now remains is to identify which totally separated compact polarities may arise as duals of doubly dense adjoint pairs which are isomorphic to ones of the form $(D^\wedge(L),D^\vee(L),u_L,l_L)$ for some lattice $L$ (not all doubly dense adjoint pairs are of this form; cf. Example~\ref{exa:notalldaDLs}).

Given any daDL $(D,E,f,g)$, there is an associated lattice $L=\im(g)\cong \im(f)$ and this lattice embeds in $D$ meet-preservingly and in $E$ join-preservingly. We write $i : L \embeds D$ and $j : L \embeds E$ for the embeddings of $L$ into $D$ and $E$, respectively. These images generate $D$ and $E$, respectively, because of the double denseness. However, the missing property is that $i$ and $j$ need not preserve admissible joins and meets, cf. Example~\ref{exa:notalldaDLs}. We will now give a dual description of this property. 

To do so, we will use the {\it canonical extension} of the adjunction $f : D \leftrightarrows E : g$ and of the embeddings $i$ and $j$. For the definition and the general theory of canonical extensions of maps we refer to \cite[Section 2]{GeJo2004}. All maps in our setting are either join- or meet-preserving, so that they are smooth and the $\sigma$- and $\pi$-extensions coincide. We therefore denote the unique extension of a (join- or meet-preserving) map $h$ by $h^\delta$. Thus, we have maps $f^\delta : D^\delta \leftrightarrows E^\delta : g^\delta$, $i^\delta : L^\delta \to D^\delta$ and $j^\delta : L^\delta \to E^\delta$. For our dual characterization, we will need the following basic fact, which is essentially the content of Remark 5.5 in \cite{GehHar2001}.
\begin{prop}\label{prop:adjlift}
Let $f : D \leftrightarrows E : g$ be an adjunction between distributive lattices and let $L$ be the lattice of Galois-closed elements. Then the following hold:
\begin{enumerate}
\item $f^\delta : D^\delta \leftrightarrows E^\delta : g^\delta$ is an adjunction;
\item The image of $g^\delta$ forms a complete $\bigwedge$-subsemilattice of $D^\delta$ which is isomorphic, as a completion of $L$, to $L^\delta$;
\item The image of $f^\delta$ forms a complete $\bigvee$-subsemilattice of $E^\delta$ which is isomorphic, as a completion of $L$, to $L^\delta$.
\end{enumerate}
\end{prop}
\begin{proof}
Item (i) is proved in \cite[Proposition 6.6]{GehHar2001}. The image of an upper adjoint between complete lattices always forms a complete $\bigwedge$-subsemilattice. To see that the image of $g^\delta$ is isomorphic to $L^\delta$ as a completion of $L$, it suffices by Theorem~\ref{thm:canextexun} to check that the natural embedding $L \embeds \im(g^\delta)$ (given by the composition $L \embeds D \embeds D^\delta$) is compact and dense. Neither of these properties is hard to verify. The proof of item (iii) is order-dual to (ii).
\end{proof}

Let $M$ be a finite subset of the lattice $L$. Recall that by Lemma~\ref{lem:joinadmcanext}, $M$ is join-admissible if and only if, for each $x\in J^\infty(L^\delta)$, we have $x\leq\bigvee M$ implies $x\leq m$ for some $m\in M$. In order to translate this to a dual condition, it is useful to get a dual characterization of the elements of $J^\infty(L^\delta)$. In the following Lemma, we will use the fact that the relation $R$ can be alternatively defined using the lifted operation $f$: regarding $X$ as $J^\infty(D^\wedge(L)^\delta)$ and $Y$ as $J^\infty(D^\vee(L)^\delta)$, we have that $x \, R \, y \iff y \leq f^\delta(x)$.

%As we've seen in Section~\ref{sec:uniform}, $X$ may be obtained as the bicompletion of a certain quasi-uniform space on  $J^\infty(L^\delta)$. On the other hand, all the elements of $X$ are $R$-closed and thus belong to $L^\delta$, see Lemma~\ref{lem:factsaboutLdelta} below. 
%Accordingly, the elements of $J^\infty(L^\delta)$ are precisely those elements of $X$ that are completely join-irreducible in $X$. That is, the ones so that $x\neq\bigvee\{x'\in X\mid x'<x\}$. We can translate this into a condition which only refers to $X$ and $R$, as we will see in the following lemma.
\begin{lem}\label{lem:factsaboutLdelta}
Let $(D,E,f,g)$ be a daDL, $(X,Y,R)$ its dual polarity, $L \cong \im(g) \cong \im(f)$, with $i : L \embeds D$ and $j : L \embeds E$ the natural embeddings. Then the following hold:
\begin{enumerate}
\item $J^\infty(D^\delta) \subseteq i^\delta[F(L^\delta)]$ and $M^\infty(E^\delta) \subseteq j^\delta[I(L^\delta)]$.
\item For all $x \in X = J^\infty(D^\delta)$, the following are equivalent:
\begin{enumerate}
\item $x \in i^\delta[J^\infty(L^\delta)]$,
\item $R[x]\neq R[\{x'\in X\mid x'<x\}]$.
\end{enumerate}
\item For all $y \in Y = J^\infty(E^\delta)$, the following are equivalent:
\begin{enumerate}
\item $\kappa(y) \in j^\delta[M^\infty(L^\delta)]$,
\item $R^{-1}[y] \neq R^{-1}[\{y' \in Y \mid y' > y\}]$.
\end{enumerate}
\end{enumerate}
\end{lem}
\begin{proof}
\begin{enumerate}
\item Let $x \in X = J^\infty(D^\delta)$. Then $x \in F(D^\delta)$, so $x$ is equal $\bigwedge F$ for some $F \subseteq D$. For each $d \in F$, since $\im(i) = \im(g)$ is join-dense in $D$, we may pick $S_d \subseteq L$ such that $d = \bigvee i(S_d)$. Let us write $\Phi$ for the set of choice functions $F \to \bigcup_{d \in F} S_d$. Then, by distributivity of $D^\delta$, we have
\[ x = \bigwedge F = \bigwedge \{\bigvee i(S_d) \ | \ d \in F\} = \bigvee \{ \bigwedge_{d \in F} i(\phi(d)) \ | \ \phi \in \Phi\}.\]
Since $x$ is completely join-irreducible in $D^\delta$, we get that $x = \bigwedge_{d \in F} i(\phi(d))$ for some $\phi \in \Phi$. Since $i^\delta$ is completely meet-preserving, we get $x = i^\delta(x')$ where $x' := \bigwedge_{d \in F} \phi(d) \in F(L^\delta)$. The proof that $M^\infty(E^\delta) \subseteq j^\delta[I(L^\delta)]$ is order-dual.
\item For the direction (a) $\Rightarrow$ (b), suppose that $R[x] = R[\{x' \in X \ | \ x' < x\}]$. %This means that, for any $y \in J^\infty(E^\delta)$, we have $y \leq f^\delta(x)$ iff there exists $x'  < x$ such that $y \leq f^\delta(x')$. 
By definition of $R$, we then get that $f^\delta(x) = \bigvee_{x' < x} f^\delta(x')$ holds in $E^\delta$. Since $f^\delta$ is lower adjoint to $g^\delta$ by Proposition~\ref{prop:adjlift}(i), we get that 
\begin{equation} \label{eq:xreduced}
x \leq g^\delta f^\delta \left(\bigvee_{D^\delta} \{x' \in X \ | \ x' < x\}\right). \tag{$\star$}
\end{equation}
By item (i), we have that $X \subseteq i^\delta[L^\delta]$, so the right-hand-side of this inequality is equal to $i^\delta(\bigvee_{L^\delta} \{ v \in (i^\delta)^{-1}(X) \ | \ i^\delta(v) < x\})$. It follows from injectivity of $i^\delta$ that if $x = i^\delta(u)$, then $u \leq \bigvee_{L^\delta} \{ v \in (i^\delta)^{-1}(X) \ | \ i^\delta(v) < x\}$. Then $u$ is actually equal to the join on the right-hand-side, so $u$ is not join-irreducible.

Conversely, if $x \not\in i^\delta[J^\infty(L^\delta)]$, then (\ref{eq:xreduced}) must hold for $x$, from which it follows that $R[x] = R[\{x' \in X \ | \ x' < x\}]$, using adjunction and the definition of $R$ again.
\item Order-dual to item (ii).\qedhere
\end{enumerate}
\end{proof}

Combining item (ii) of this Lemma with the characterization of join-admissibility in Lemma~\ref{lem:joinadmcanext}, we now get the following. A finite set $M \subseteq L$ being join-admissible corresponds to saying that, for each $x\in X$ with $R[x]\neq R[\{x'\in X\mid x'<x\}]$, we have $R[x]\subseteq R\left[\bigcup\{\widehat{m}\mid m\in M\}\right]$ implies $x\in\bigcup\{\widehat{m}\mid m\in M\}$. Note that in its dual incarnation this property does not really depend on $M$ but only on the clopen down-set $\bigcup\{\widehat{m}\mid m\in M\}$. 
Accordingly, we make the following definition.

\begin{dfn}
Let $(X,Y,R)$ be a TSCP, and $U\subseteq X$ a clopen down-set. We say that $U$ is \emph{$R$-regular} provided that, for each $x\in X$ with $R[x]\neq R[\{x'\in X\mid x'<x\}]$, we have $R[x]\subseteq R[U]$ implies $x\in U$. 

Order dually, we say that a down-set $V \subseteq Y$ is \emph{$R$-coregular} provided that, for each $y \in Y$ with $R^{-1}[y] \neq R^{-1}[y' \in Y \mid y' > y]$, we have $R^{-1}[y] \subseteq R^{-1}[U]$ implies $y \in U$.
\end{dfn}
Recall that a clopen down-set $U\subseteq X$ is $R$-closed provided that, for each $x\in X$, we have $R[x]\subseteq R[U]$ implies $x\in U$. Thus it is clear that every $R$-closed clopen down-set in $X$ is $R$-regular. Preserving admissible joins exactly corresponds to the reverse implication: as soon as $U$ is $R$-regular it must also be $R$-closed. To sum up:
\begin{prop}\label{prop:dualcharlattice}
Let $(D,E,f,g)$ be a daDL, and let $(X,Y,R)$ be its dual polarity. Then the following are equivalent:
\begin{enumerate}
\item There exists a lattice $L$ such that $(D,E,f,g) \cong (D^\wedge(L),D^\vee(L),u_L,l_L)$;
\item The embedding $\im(g) \embeds D$ preserves admissible joins and the embedding $\im(f) \embeds E$ preserves admissible meets.
\item In $(X,Y,R)$, all $R$-regular clopen downsets in $X$ are $R$-closed, and all $R$-coregular clopen downsets in $Y$ are $R$-open.
\end{enumerate}
\end{prop}
\begin{proof}
The equivalence (i) $\iff$ (ii) holds by the results in Section~\ref{sec:envelopes}. 

Throughout the proof of the equivalence (ii) $\iff$ (iii), let us write $L$ for the lattice $\im(g)$, in which meets are given as in $D$ and $\bigvee_L S = gf(\bigvee_M S)$, for any $S \subseteq L$. In this proof, we regard $L$ as a sublattice of $D$, suppressing the notation $i$ for the embedding $L \embeds D$.

For the implication (ii) $\Rightarrow$ (iii), let $U$ be an $R$-regular clopen downset in $X$. Since $\im(g)$ is dense in $D$, there exists $M \subseteq \im(g)$ such that $U = \bigcup_{m \in M} \hat{m}$. We show that $M$ is join-admissible in the lattice $L$, using Lemma~\ref{lem:joinadmcanext}. If $x \in J^\infty(L^\delta)$ and $x \leq \bigvee_L M = gf(\bigvee_D M)$, then $f^\delta(x) \leq f(\bigvee_D M)$. So, by definition of $R$ and the fact that $f$ is completely join-preserving, we get that $R[x] \subseteq R[\bigcup_{m \in M} \hat{m}] = R[U]$. Since $U$ is $R$-regular and $x \in J^\infty(L^\delta)$, we get that $x \in U$, so $x \leq m$ for some $m \in M$. So $M$ is join-admissible, so (ii) implies that $\bigvee_L M = gf(\bigvee_D M) = \bigvee_D M$. That is, $\overline{U} = U$, so $U$ is $R$-closed. The proof that $R$-coregular clopen downsets in $Y$ are $R$-open is dual.

For the implication (iii) $\Rightarrow$ (ii), let $M \subseteq L$ be a join-admissible subset. Let $U := \bigcup_{m \in M} \hat{m} \subseteq X$. Then $U$ is clearly a clopen downset. We show that $U$ is $R$-regular. Let $x \in X$ such that $R[x]\neq R[\{x'\in X\mid x'<x\}]$ and $R[x] \subseteq R[U]$. Then $x \in J^\infty(L^\delta)$ and $f^\delta(x) \leq f(\bigvee_D M)$, so $x \leq gf(\bigvee_D M) = \bigvee_L M$. So, since $M$ is join-admissible, there exists $m \in M$ such that $x \leq m$. In particular, we have $x \in U$, as required. By the assumption (iii), we conclude that $U$ is $R$-closed, i.e., $\overline{U} = U$, so that $\bigvee_L M = gf(\bigvee_D M) = \bigvee_D M$. The proof that $\im(f) \embeds E$ preserves admissible meets is dual.
\end{proof}

In the light of this proposition, we can now define a subcategory of TSCP's which will be dual to the category of lattices with admissible homomorphisms.
\begin{dfn}\label{dfn:tight}
Let $(X,Y,R)$ be a TSCP. We say that $(X,Y,R)$ is {\it tight} if all $R$-regular clopen downsets in $X$ are $R$-closed, and all $R$-coregular clopen downsets in $Y$ are $R$-open. We denote by $\mathbf{tTSCP}$ the full subcategory of $\mathbf{TSCP}$ whose objects are the tight TSCP's.
\end{dfn}

We now obtain our topological duality theorem for lattices with admissible homomorphisms.
\begin{thm}\label{thm:mainduality}
The category $\mathbf{L_a}$ of lattices with admissible homomorphisms is dually equivalent to the category $\mathbf{tTSCP}$ of tight totally separated compact polarities.
\end{thm}
\begin{proof}
By Proposition~\ref{prop:latticesasdaDLs}, we have that $\mathbf{L_a}$ is equivalent to a full subcategory of $\mathbf{daDL}$. By Theorem~\ref{thm:dualitydaDL}, the category $\mathbf{daDL}$ is dually equivalent to $\mathbf{TSCP}$. By Proposition~\ref{prop:dualcharlattice}, the image of $\mathbf{L_a}$ in $\mathbf{daDL}$ under this dual equivalence is $\mathbf{tTSCP}$.
\end{proof}
The above theorem is not as general as possible: although we have only developed a duality for $\mathbf{L_a}$ here, it should be possible to generalize this duality to the categories $\Lwav$ and $\Lvaw$. To do so, one would need to generalize the category $\mathbf{tTSCP}$ to one where the morphisms are single functions instead of pairs of functions. We leave this to future work.

In this section, in light of Examples~\ref{exa:notsober}~and~\ref{exa:notspectral}, we set out to obtain a topological duality for lattices in which the spaces are nicer than those occurring in Hartung's duality. 
Although the spaces obtained in our duality are as nice as can be (they are compact, Hausdorff and totally disconnected), this comes at the price of a rather complicated characterization. Therefore, we are inclined to draw as a negative conclusion that topology may not be the most opportune language to discuss `duality' for lattices (unless the definition of a tTSCP can be simplified). Fortunately, the perspective of canonical extensions provides an alternative to topology: we have explained above how canonical extensions can be viewed as a point-free version of Hartung's duality, and we have used them to reason about the topological dual spaces introduced in this paper. In the next section, we will propose quasi-uniform spaces as a useful ``spatial'' alternative to topology in the context of set-theoretic representations of lattices.

%SECTION ON QUASI-UNIFORM SPACES

\section{Quasi-uniform spaces associated with a lattice}\label{sec:uniform}
In this section we will show that the distributive envelopes of a lattice, which were defined by a universal property in Section~\ref{sec:envelopes}, are also natural from a generalized topological perspective. The appropriate framework is that of quasi-uniform spaces, which generalize both quasi-orders and topologies (see \cite{FletcherLindgren82}, in particular Chapter 3, for background on the theory of quasi-uniform spaces used in this section). In this section we will associate two Pervin quasi-uniform spaces to a lattice $L$, and then show in Theorem~\ref{thm:unifdual} that the completions of these quasi-uniform spaces coincide with the dual spaces of the distributive envelopes of $L$. Thus, quasi-uniform spaces give a precise spatial meaning to the distributive envelopes of $L$. Note that Pervin spaces, uniform completions and compactifications have already been used by Ern\'e and Palko \cite{Ern2001, ErnPal1998} to obtain order-theoretic ideal completions.

Given a set $X$, we denote, for each subset $A\subseteq X$, by $U_A$ the subset 
\[
(A^c\times X)\cup(X\times A)=\{(x,y)\mid x\in A\implies y\in A\}
\]
of $X\times X$. Given a topology $\tau$ on $X$, the filter $\mathcal U_\tau$ in the power set of $X\times X$ generated by the sets $U_A$ for $A\in\tau$ is a totally bounded transitive quasi-uniformity on $X$ \cite[Proposition 2.1]{FletcherLindgren82}. The quasi-uniform spaces $(X,\mathcal U_\tau)$ were first introduced by Pervin \cite{Pervin62} and are now known in the literature as Pervin spaces. 
Generalizing this idea (also see \cite{Csa1993}), given any subcollection $\mathcal K\subseteq \mathcal P(X)$, we define $(X,\mathcal U_{\mathcal K})$ to be the quasi-uniform space whose quasi-uniformity is the filter generated by the entourages $U_{A}$ for $A\in \mathcal K$. Here we will call this larger class of quasi-uniform spaces {\it Pervin spaces}. 

The first crucial point is that, for any collection $\mathcal K\subseteq \mathcal P(X)$, the bounded distributive sublattice $D(\mathcal K)$ of $\mathcal P(X)$ generated by $\mathcal K$ may be recovered from $(X,\mathcal U_{\mathcal K})$, even though this cannot be done in general from the associated topology. The {\it blocks} of a space $(X,\mathcal U)$ are the subsets $A\subseteq X$ such that $U_A$ is an entourage of the space, or equivalently, those for which the characteristic function $\chi_A:X\to 2$ is uniformly continuous with respect to the Sierpi\'nski quasi-uniformity on $2$, which is the one containing just $2^2$ and $\{(0,0),(1,1),(1,0)\}$. The following fact is well-known, but we give a proof since it does not seem to be readily available in the literature.
\begin{thm}\label{thm:pervinblocks}
Let $X$ be a set and $\mathcal{K} \subseteq \Po(X)$ a collection of subsets. The blocks of the quasi-uniform space $(X,\mathcal U_{\mathcal K})$ are exactly the elements of the sublattice $D(\mathcal{K})$ of $\Po(X)$ generated by $\mathcal{K}$.
\end{thm}
\begin{proof}
The blocks of any quasi-uniform space form a lattice, since $U_A \cap U_B$ is contained in both $U_{A \cap B}$ and $U_{A \cup B}$, for any $A, B \subseteq X$. If $A$ is a block of $\mathcal{U}_{\mathcal{K}}$, then by definition $U_A$ contains a set of the form $\bigcap_{B \in \mathcal{F}} U_B$, where $\mathcal{F} \subseteq \mathcal{K}$ is finite. From $\bigcap_{B \in \mathcal{F}} U_B \subseteq U_A$, it follows easily that $A = \bigcup \{ \bigcap \{ B \ | \ x \in B, B \in \mathcal{F} \} \ | \ x \in A\}$ (cf. \cite[Lemma 2]{Csa1993}). %Note that the union on the right hand side is finite because $\mathcal{F}$ is.
\end{proof}

Further, it is not hard to see that if $\mathcal D\subseteq \mathcal P(Y)$ and $\mathcal E\subseteq \mathcal P(X)$ are bounded sublattices of the respective power sets, then a map $f:(X,\mathcal U_{\mathcal E})\to(Y,\mathcal U_{\mathcal D})$ is uniformly continuous if and only if $f^{-1}$ induces a lattice homomorphism from $\mathcal D$ to $\mathcal E$ by restriction. Thus, the category of sublattices of power sets with morphisms that are commuting diagrams

\vspace{-1em}

\begin{center}
\begin{tikzpicture}[scale=0.8]
\matrix (m) [matrix of math nodes, row sep=3em, column sep=2.5em, text height=1.5ex, text depth=0.25ex] 
{ \mathcal{D} & \mathcal{E} \\
\Po(Y) &  \Po(X) \\};
\path[->] (m-1-1) edge node[above] {$h$} (m-1-2);
\path[->] (m-2-1) edge node[above] {$\phi$} (m-2-2);
\path[>->] (m-1-1) edge (m-2-1);
\path[>->] (m-1-2) edge (m-2-2);
\end{tikzpicture}
\end{center}

\vspace{-1em}

where $\phi$ is a complete lattice homomorphism, is dually isomorphic to the category of Pervin spaces with uniformly continuous maps.

To be able to state the main result from \cite{GGP2010} that we want to apply here, we need to recall the definition of {\it bicompletion} of a quasi-uniform space. For more details see \cite[Chapter 3]{FletcherLindgren82}. Bicompleteness  generalizes the notion of completeness for uniform spaces, which is well-understood (see, e.g., \cite[Chapter II.3]{Bourbaki1}): a uniform space $(X,\mathcal{U})$ is {\it complete} if every Cauchy filter converges. %Here, we recall that a proper filter $\mathcal{F}$ of $\Po(X)$ is {\it Cauchy} if every entourage $U \in \mathcal{U}$ contains a set of the form $F \times F$ for some $F \in \mathcal{F}$, and the filter $\mathcal{F}$ is said to {\it converge} to $x \in X$ if $U(x) \in \mathcal{F}$ for all $U \in \mathcal{U}$. In a uniformity that comes from a metric space, Cauchy filters correspond exactly to Cauchy sequences.
Now let $(X,\mathcal U)$ be a quasi-uniform space. A quasi-uniform space $(X,\mathcal U)$ is called {\it bicomplete} if and only if its symmetrization $(X,\mathcal U^s)$ is a complete uniform space. Here, recall that the {\it symmetrization}, $\mathcal U^s$, of the quasi-uniformity $\mathcal{U}$ is defined as the filter of $\Po(X \times X)$ generated by the union of  $\mathcal U$ and $\mathcal U^{-1}$. 
It has been shown by Fletcher and Lindgren \cite[Chapter 3.3]{FletcherLindgren82} that the full subcategory of bicomplete quasi-uniform spaces forms a reflective subcategory of the category of quasi-uniform spaces with uniformly continuous maps. Thus, for each quasi-uniform space $(X,\mathcal U)$, there is a bicomplete quasi-uniform space $(\widetilde{X},\widetilde{\mathcal U})$ and a uniformly continuous map $\eta_X:(X,\mathcal U)\to(\widetilde{X},\widetilde{\mathcal U})$ with an appropriate universal property.

Now we are ready to state the main result of Section 1 of \cite{GGP2010}: %The set representation of a distributive lattice $D$ given by Stone/Priestley duality is obtainable from {\it any} bounded lattice embedding $e:D\hookrightarrow \mathcal P(X)$ of $D$ into a power set by taking the {\it bicompletion} of the corresponding quasi-uniform Pervin space $(X,\mathcal U_{Im(e)})$. To be more precise, we have:
% This is actually a generalization of that result, from BA to DL!
\begin{thm}\cite[Theorem 1.6]{GGP2010}\label{thm:completion}
Let $D$ be a bounded distributive lattice, and let $e:D\hookrightarrow \mathcal P(X)$ be any bounded lattice embedding of $D$ in a power set lattice and denote by $\mathcal D$ the image of the embedding $e$. Let $\widetilde{X}$ be the bicompletion of the Pervin space $(X,\mathcal U_{\mathcal D})$. Then $\widetilde{X}$ with the induced topology is the Stone dual space of $D$.
\end{thm}
Alternatively, one can think of the quasi-uniform space $(\widetilde{X}, \widetilde{\mathcal{U}}_{\mathcal{D}})$ as an ordered uniform space, as follows. Equip the uniform space $(\widetilde{X},\widetilde{\mathcal U^s_{\mathcal D}})$ with the order $\leq$ defined by $\bigcap_{a\in {\mathcal D}}U_{\widehat{a}}$. Then $(\widetilde{X},\widetilde{\mathcal U^s_{\mathcal D}},\leq)$ is a uniform version of the {\it Priestley dual space} of $D$.

We now want to apply this theorem to the setting of this paper. Let $L$ be a bounded lattice, $L^\delta$, the canonical extension of $L$, and $X_L=J^\infty(L^\delta)$ and $Y_L=M^\infty(L^\delta)$. Then $L$ induces quasi-uniform space structures $(X_L,\mathcal U_{\hat{L}})$ and $(Y_L,\mathcal U_{\check{L}})$ on $X_L$ and $Y_L$, respectively. Here $\mathcal U_{\hat{L}}$ is the Pervin quasi-uniformity generated by the image $\hat{L}=\{\hat{a}\mid a\in L\}$ and $\mathcal U_{\check{L}}$ is the Pervin quasi-uniformity generated by the image $\check{L}=\{\check{a}\mid a\in L\}$. By Theorem~\ref{thm:completion}, the bicompletions of these Pervin spaces are Stone spaces and the corresponding bounded distributive lattices are the sublattices of $\mathcal P(X_L)$ and $\mathcal P(Y_L)$ generated by $\hat{L}$ and $\check{L}$, respectively.

The following theorem now follows by combining Proposition~\ref{prop:isodwedge}, Theorem~\ref{thm:envelopeexists} and Theorem~\ref{thm:completion}.

\begin{thm}\label{thm:unifdual}
Let $L$ be a lattice. The bicompletion of the associated quasi-uniform Pervin space, $(X_L,\mathcal U_{\hat{L}})$, is the dual space of the distributive $\wedge$-envelope, $D^\wedge(L)$, of $L$. Order dually, the bicompletion of the quasi-uniform Pervin space $(Y_L,\mathcal U_{\check{L}})$ is the dual space of the distributive $\vee$-envelope, $D^\vee(L)$, of $L$.
\end{thm}

\begin{exa}
For any finite lattice $L$, the distributive envelope $D^\wedge(L)$ is the lattice of downsets of the poset $J(L)$, with the order inherited from $L$. Thus, in the finite case, the quasi-uniform space $X_L$ is already bicomplete, and hence equal to its own bicompletion. The same of course holds for $D^\vee(L)$ and $Y_L$. In the finite case, $X_L$ and $Y_L$ are just the spaces occurring in Hartung's duality.

For the lattice $L$ discussed in Example~\ref{exa:notsober} above, the distributive envelope $D^\wedge(L)$ is (isomorphic to) the lattice consisting of all finite subsets of the countable antichain, and a top element. Thus, in the bicompletion of $X_L$, we find one new point, corresponding to the prime filter consisting of only the top element.

For the lattice $K$ discussed in Example~\ref{exa:notspectral}, the distributive envelope $D^\wedge(K)$ is a much bigger lattice than $K$, and the bicompletion of $X_L$ will contain many new points. In particular, the bicompletion will not just be the soberification of $X_L$.
\end{exa}

%CONCLUDING SECTION

\section*{Conclusion}

In this paper, we developed the theory of distributive envelopes and used it to obtain a topological duality for lattices. We see our methodology as an example of the phenomenon that canonical extensions and duality may help to study lattice-based algebras, even when they do not lie in finitely generated varieties. As a case in point, the construction of the distributive envelopes in Section~\ref{sec:envelopes} made use of canonical extensions of lattices as a key tool. Moreover, the work in that section enabled us to identify the $(\wedge,a\vee)$-morphisms between lattices, which are exactly the ones which have functional duals on the $X$-components of the dual spaces defined in Section~\ref{sec:duality}. We believe that canonical extensions may be used in a similar way for other varieties of algebras based on lattices, such as residuated lattices, to mention just one example.

In Section~\ref{sec:uniform}, we provided an alternative view of set-representation of lattices, which replaces topology by quasi-uniformity and completion. Theorem~\ref{thm:unifdual} opens the way for obtaining an alternative duality for lattices, in which quasi-uniform spaces take the place of topological spaces. To do so, an interesting first step would be to represent the adjunction $D^\wedge(L) \leftrightarrows D^\vee(L)$ as additional structure on the pair of quasi-uniform spaces. We leave the development of these ideas to further research.

Let us mention one more possible direction for further work. For distributive lattices, the canonical extension functor is left adjoint to the inclusion functor of perfect distributive lattices into distributive lattices. However, this is known to be true for lattice-based algebras only in case all basic operations are both Scott and dually Scott continuous (see \cite[Proposition~C.9, p. 196]{Coumans12} for a proof in the distributive setting). It follows from the results in Goldblatt \cite{Goldblatt2006} that the canonical extension functor for modal algebras (i.e., Boolean algebras equipped with a modal operator) can be viewed as a left adjoint. However, the codomain category that is involved here is not immediately obvious: it is not the category of `perfect modal algebras' in the usual sense. We conjecture that the distributive envelope constructions developed in Section~\ref{sec:envelopes} of this paper may be used to define a category in which the canonical extension for lattices is a left adjoint. We also leave the actual development of this line of thought to future research.

\bibliographystyle{amsplain}

\bibliography{latticebib}

\end{document}